\documentclass[11pt]{amsart}

\usepackage[margin=1.05in]{geometry}
\usepackage{amsmath,amssymb,amsthm,mathtools}
\usepackage{microtype}
\usepackage{hyperref}
\usepackage[nameinlink,capitalise]{cleveref}

\numberwithin{equation}{section}

\hypersetup{colorlinks=true,linkcolor=blue,citecolor=blue,urlcolor=blue}

\newtheorem{theorem}{Theorem}[section]
\newtheorem{proposition}[theorem]{Proposition}
\newtheorem{lemma}[theorem]{Lemma}
\theoremstyle{definition}
\newtheorem{remark}[theorem]{Remark}

\newcommand{\Sn}{\mathbb S^{n-1}}
\newcommand{\tr}{\operatorname{tr}}
\newcommand{\osc}{\operatorname{osc}}
\newcommand{\dist}{\operatorname{dist}}
\newcommand{\supp}{\operatorname{supp}}
\newcommand{\Gtwo}{\Gamma_2}
\newcommand{\Slin}{\mathcal S}
\newcommand{\dd}{d_2}
\newcommand{\ip}[2]{\left[#1,#2\right]}
\newcommand{\norm}[1]{\left\lVert #1\right\rVert}
\newcommand{\abs}[1]{\left|#1\right|}
\newcommand{\eps}{\varepsilon}

\let\cal\mathcal

\title{Interior \(C^{2,\alpha}\) Regularity for the Quadratic Hessian Equation}
\author{Ruosi Chen, Xingchen Zhou and Ruixuan Zhu}

\address{Ruosi Chen, Department of Mathematical Sciences, Tsinghua University, Beijing, 100084, China.}
\email{crs22@mails.tsinghua.edu.cn}

\address{Xingchen Zhou, School of Mathematics and Statistics, Hainan University, Haikou, 570228, China.}
\email{zxc3zxc4zxc5@stu.xjtu.edu.cn}

\address{Ruixuan Zhu, Institute for Theoretical Sciences, Westlake University, Hangzhou, 310030, China.}
\email{zhuruixuan@westlake.edu.cn}
\date{}

\begin{document}

\begin{abstract}
We establish interior \(C^{2,\alpha}\) regularity for admissible solutions of the quadratic Hessian equation with positive \(C^\alpha\) right-hand side on the full positive branch.
The main ingredients are a quantitative large-trace propagation argument and an adaptive Dirichlet comparison method.
\end{abstract}


\maketitle

\section{Introduction}

We study the interior $C^{2,\alpha}$ regularity of 2-convex solutions to
\begin{equation}
 \sigma_2(D^2u)=f>0,
  \label{eq:main-pde}
\end{equation}
in general dimensions.
The equation is considered on the standard G\aa rding cone
\[
 \Gtwo=\{A\in\operatorname{Sym}(n):
             \sigma_1(A)>0,\ \sigma_2(A)>0\}.
\]
On this branch the function
\[
 G(A)=\sqrt{\sigma_2(A)},\qquad A\in\Gtwo,
\]
is elliptic, concave, and homogeneous of degree one, and \eqref{eq:main-pde} is equivalent to \(G(D^2u)=g:=\sqrt f\).
\begin{theorem}\label{thm:smooth}
Let \(n\ge2\) and \(0<\alpha<1\).
Suppose that $u\in C^\infty(B_2)$ satisfies the equation
\[
 \sigma_2(D^2u)=f\quad\text{in }B_2, \qquad D^2u(x)\in\Gtwo,
 \]
where \(f\in C^\alpha(B_2)\) and \(f\ge f_0>0\).
Then
\begin{equation}
 \;
 \norm{u}_{C^{2,\alpha}(B_{1/2})}
 \le C\!\left(n,\alpha,f_0,\norm{f}_{C^\alpha(B_2)}, \norm{u}_{L^\infty(B_2)}\right).\;
 \label{eq:smooth-estimate}
\end{equation}
\end{theorem}

We use the convention of a standard class \(\Phi_2(\Omega)\) of continuous \(2\)-convex functions introduced in \cite{TW1997,TW1999}.
A continuous function \(u\) belongs to \(\Phi_2(\Omega)\) if it is locally a uniform limit of smooth functions whose Hessians lie in \(\overline\Gtwo\).
Equivalently, every \(C^2\) function touching \(u\) from above at \(x\) has Hessian in \(\overline\Gtwo\) at that point.
For \(f \in C(\Omega)\), a function \(u\in\Phi_2(\Omega)\) is an admissible viscosity solution of \(\sigma_2(D^2u)= f\) if every upper test \(\phi\) satisfies \( \sigma_2(D^2\phi(x))\ge f(x), \) and every lower test which is \(2\)-convex near the contact point satisfies the reverse inequality.

\begin{theorem}\label{thm:viscosity}
Let \(u\in C(B_2)\cap\Phi_2(B_2)\) be an admissible viscosity solution of
\[
 \sigma_2(D^2u)=f\in C^\alpha(B_2),\qquad f\ge f_0>0.
\]
Then
\[
 u\in C^{2,\alpha}_{\mathrm{loc}}(B_2).
\]
\end{theorem}

When \(n=2\), equation \eqref{eq:main-pde} is the Monge--Amp\`ere equation and admissibility already implies convexity; the classical interior Hessian estimate goes back to Heinz \cite{Heinz1959}.
In dimension three, Warren--Yuan obtained the estimate for the constant equation by using its special Lagrangian structure \cite{WarrenYuan2009}, and Qiu later treated variable right sides \cite{Qiu2024}.
Shankar--Yuan resolved the unrestricted constant equation in dimension four and, in higher dimensions, proved an estimate under a dynamic semiconvexity condition \cite{ShankarYuan2025}.
Fan extended the result in dimension four, and the conditional result in higher dimensions, to positive \(C^{1,1}\) right sides \cite{Fan2026}.

In arbitrary dimension, a second line of work imposes a condition on the negative eigenvalues.
Guan--Qiu obtained estimates under a lower bound for \(\sigma_3(D^2u)\), which includes convex solutions \cite{GuanQiu2019}.
McGonagle--Song--Yuan treated almost convex solutions by compactness \cite{McGonagleSongYuan2019}, while Shankar--Yuan used a Jacobi inequality and the Legendre--Lewy transform for semiconvex solutions \cite{ShankarYuan2020} and established regularity for almost convex viscosity solutions \cite{ShankarYuan2021}.
Mooney proved strict \(2\)-convexity and regularity for convex viscosity solutions \cite{Mooney2021}; more recently, he obtained an interior Hessian estimate for the constant equation in terms of a \(W^{2,p}\) norm, for every \(p>2\) \cite{Mooney2025}.

Several recent papers are especially close to the present problem.
Chen--Jian--Tu--Zhou proved interior \(C^2\) regularity for convex viscosity solutions with positive Lipschitz right side \cite{ChenJianTuZhou2026}.
Zhou--Zhu obtained the exact \(C^{2,\alpha}\) conclusion for convex admissible viscosity solutions with positive \(C^\alpha\) right side \cite{ZhouZhu2026}.
On the other hand, Li--Wu removed the convexity and dimension restrictions for a constant right-hand side and derived a Hessian estimate on the positive branch \cite{LiWu2026}.
Their argument also indicates a \(C^{1,1}\) extension to a variable right side whose estimate depends on a \(C^1\) norm of the solution.
Taken together, these results treat either H\"older right sides under convexity or the constant equation on the unrestricted positive branch.
Our main result extends the Hessian estimate on the unrestricted positive branch to positive \(C^\alpha\) right-hand sides.

Two obstacles arise in proving this result.
Before the Hessian is bounded, the equation need not be uniformly elliptic, so the standard estimates for concave equations cannot be applied.
A direct second derivative calculation differentiates the right-hand side and is therefore incompatible with \(f\in C^\alpha\).
Smoothing \(f\) does not by itself solve the problem, because the first two derivatives of the smoothed function diverge as the smoothing scale tends to zero.

We now describe the proof.
Assuming that the trace is unbounded, a logarithmic stopping argument, resolvent averaging, and the local maximum principle produce a measurable set \(E\) such that
\[
 \Delta u\ge cM\quad\hbox{on }E,\qquad
 \abs E\ge cR^{n-1}h.
\]
This construction is quantified in \cref{lem:stopping,prop:large-trace-set}.
On a fixed ball \(D\), we construct an adaptive comparison solution \(v\) with the same boundary values as \(u\), solving
\[
 G(D^2v)+\zeta(x)(g-b_\eps)
       \eta\!\left(\frac{\Delta v}{P}\right)=g.
\]
Here \(b_\eps\le g\) is a smooth lower approximation, \(\zeta\) localizes the comparison, \(\eta\) is a trace cutoff, and \(P\) is the threshold; the precise choices and the Dirichlet problem are given in \cref{sec:dirichlet-problem}.
Where \(\zeta=1\), this equation interpolates between the original H\"older data and \(b_\eps\), becoming the smooth equation \(G(D^2v)=b_\eps\) once the trace reaches the threshold.
We then get 
\[
 \int_D\dd(D^2u,D^2v)^2
 \le C\frac{\norm{g-b_\eps}_{L^\infty(D)}^2}{P};
\]
see \cref{lem:reverse-cs,lem:comparison-l2}.
Consequently the large-trace bound transfers from \(u\) to \(v\) on a fixed fraction of \(E\), as stated in \cref{prop:comparison-lower-trace}.

On the resulting comparison component, a gradient estimate and the Pogorelov-type calculation of \cite[Sections~3--4]{ChouWang2001} give the upper trace estimate in \cref{prop:comparison-upper-trace}.
The trace cutoff forces the maximum in the Pogorelov argument to lie in the region where \(v\) solves \(G(D^2v)=b_\eps\), hence the calculation differentiates only the smooth function \(b_\eps\).
The choices
\[
 \eps=M^{-1/8},\qquad P=M^{1-\alpha/8}
\]
from \eqref{eq:parameter-choice} make both the comparison error and the differentiated smooth data become \(o(M)\), giving a contradiction and hence the Hessian is bounded.
Standard estimates for concave uniformly elliptic equations then give the exponent \(\alpha\).
Finally, the Hessian-measure comparison theorem passes the smooth estimate to continuous \(2\)-convex viscosity solutions.

The paper is organized as follows.
Section~\ref{sec:large-trace} turns a large trace at one point into a large-trace set of positive measure.
Section~\ref{sec:comparison} constructs the comparison solution and proves its lower and upper trace estimates.
Section~\ref{sec:contradiction} combines these estimates to prove the Hessian and \(C^{2,\alpha}\) bounds, and then proves Theorem~\ref{thm:viscosity} by smooth approximation.
\section{From a large trace to a set of positive measure}
\label{sec:large-trace}

We use
\[
 \langle A,B\rangle:=\operatorname{tr}(A^T B)=\sum_{i,j}A_{ij}B_{ij}
\]
for the Frobenius inner product on \(\operatorname{Sym}(n)\). Let $\abs{A}=\sqrt{\langle A,A\rangle}$ denote the Frobenius norm of $A$.
We first record the elementary facts about \(\Gtwo\) used throughout this and the next section.

\begin{lemma}\label{lem:newton}
For \(A\in\Gtwo\), we have $(\tr A)\mathrm{I}-A>0$, and 
\begin{equation}
 \abs{A}^2=(\tr A)^2-2\sigma_2(A)<(\tr A)^2.
\label{eq:newton-positive}
\end{equation}
In particular,
\begin{equation}
 \lambda_{\max}(A)<\tr A,\qquad
 \lambda_{\min}(A)>-(n-2)\tr A.
\label{eq:eigenvalue-bounds}
\end{equation}
Moreover
\begin{equation}
 DG(A)=\frac{(\tr A)\mathrm{I}-A}{2G(A)},
\label{eq:linearization}
\end{equation}
and \(G\) is concave and homogeneous of degree one on \(\Gtwo\).
\end{lemma}

\begin{proof}
On \(\Gtwo\), the matrix \(\partial\sigma_2/\partial A=(\tr A)\mathrm{I}-A\) is positive definite.
The identity in \eqref{eq:newton-positive} implies \(\abs{\lambda_i}<\tr A\), and summing the remaining eigenvalues gives the lower bound in \eqref{eq:eigenvalue-bounds}.
Formula \eqref{eq:linearization} follows by the chain rule.
The concavity of \(\sigma_2^{1/2}\) is the standard G\aa rding inequality.
\end{proof}

In order to obtain the Hessian bound of $u$, it suffices to prove the $L^{\infty}$ bound of $\Delta u$. 
We prove this by contradiction. 
We now begin the large-trace propagation argument, which uses no global \(C^1\) bound.

\begin{lemma}\label{lem:stopping}
Let \(u_j\) be smooth admissible solutions of
\[
 \sigma_2(D^2u_j)=f_j\quad\text{in }B_2,
\]
such that
\begin{equation}
 0<f_0\le f_j\le f_1,\qquad
 [f_j]_{C^\alpha(B_2)}\le C,
 \qquad
 \sup_{B_{1/2}}\Delta u_j\longrightarrow\infty.
\label{eq:blowup-assumptions}
\end{equation}
Then after passing to a subsequence, there exist
$M_j\to \infty$, $R_j\to 0$ and \(x_j\in B_{5/8}\) such that
\begin{equation}
 \Delta u_j(x_j)=M_j,\qquad
 0<\Delta u_j\le2M_j\quad\text{on }B_{R_j}(x_j).
\label{eq:local-trace-cap}
\end{equation}
In particular, we take
\begin{equation}
 R_j=\frac{\ell_j}{F(M_j)},
\label{eq:stopping-scales}
\end{equation}
where \(F(s)=(\log(e+s))^2\), and 
\begin{equation}
\ell_j\to\infty,\qquad \frac{\ell_j}{\log(e+M_j)} \to 0.
\end{equation}
\end{lemma}

\begin{proof}
Choose \(y_j\in B_{1/2}\) and set \(L_j=\Delta u_j(y_j)\to\infty\).
Choose \(\ell_j\to\infty\) such that \(\ell_j / \log(e+ L_j) \to 0\).
Set \(x_{j,0}=y_j\) and \(M_{j,0}=L_j\). 
Inductively, for $k \in \mathbb{N}$, if
\[
 \Delta u_j\le 2M_{j,k} \quad\text{on }B_{\ell_j/F(M_{j,k})}(x_{j,k}),
\]
we terminate the construction. 
Otherwise, we choose
\[
 x_{j,k+1}\in B_{\ell_j/F(M_{j,k})}(x_{j,k})
 \quad\text{such that}\quad 
 M_{j,k+1}:=\Delta u_j(x_{j,k+1})>2M_{j,k}.
\]
Thus, as long as the construction continues, we have $ M_{j,k}\ge 2^kM_{j,0} = 2^k L_j $ and 
\[
 \abs{x_{j,k}-y_j}
 \le \sum_{m=0}^{k-1}\frac{\ell_j}{F(M_{j,m})}
 \le \sum_{m\ge0}\frac{\ell_j}{F(2^mL_j)}
 \le \frac{C\ell_j}{\log(e+L_j)}=o(1),
\]
where the last inequality follows from 
\[
 \sum_{m \geq 0} \frac{1}{F(2^m L_j)}
 \leq C \sum_{m \geq 0} (a_j + m)^{-2}
 \leq C \left( a_j^{-2} + \int_0^{\infty} (a_j + t)^{-2} \, dt \right)
 \leq \frac{C}{a_j},
\]
where \(a_j=\log(e+L_j)\).
Thus, after discarding finitely many indices $j$, every ball used in the construction is contained in $B_{5/8}$.

For fixed \(j\), the construction must terminate: otherwise \(M_{j,k}\ge2^kL_j\) at points in the compact set \(\overline{B}_{5/8}\), contradicting the boundedness of the smooth function \(\Delta u_j\) in \(\overline{B}_{5/8}\). 
Let \(k_j\) be the terminal index and set
\[
 x_j=x_{j,k_j},\qquad M_j=M_{j,k_j},\qquad
 R_j=\frac{\ell_j}{F(M_j)}.
\]
This gives \eqref{eq:local-trace-cap}.
Moreover, \(M_j\ge L_j\to\infty\), and hence
\[
 \frac{\ell_j}{\log(e+M_j)}
 \le \frac{\ell_j}{\log(e+L_j)}\longrightarrow0,
 \qquad R_j\longrightarrow0.
\]
\end{proof}

The main purpose of this section is to prove the following proposition.
\begin{proposition}\label{prop:large-trace-set}
Let \(u_j,x_j,M_j,\ell_j,R_j\) be as in \cref{lem:stopping}.  After passing to a subsequence, there exist
\(h_j>0\) and measurable sets \(E_j\subset B_{R_j}(x_j)\) such that
\[
\Delta u_j\ge cM_j\quad\text{on }E_j,
\]
where
\[
\abs{E_j}\ge cR_j^{n-1}h_j,\qquad
h_j\sim\frac{1}{M_jF(M_j)}.
\]
\end{proposition}

To prove \cref{prop:large-trace-set}, we first use the Bessel resolvent to find a direction $e_j$ and a scale \(h_j\) along which the second difference of \(u_j\) is of order \(M_j\) on a small ball.  We then propagate this lower bound in the \(n-1\) transverse directions to produce a set \(E_j\) of the required measure.

\subsection{A large directional second difference}

For a unit vector \(e\), define
\[
 Q_h^eu(x)=\frac{u(x+he)+u(x-he)-2u(x)}{h^2}.
\]

\begin{lemma}\label{lem:bessel}
Under the hypotheses and notation of \cref{lem:stopping}, 
there exist scales \(h_j\sim\frac1{M_jF(M_j)}\), points
\(z_j\in B_{R_j/4}(x_j)\), and directions \(e_j\in\mathbb S^{n-1}\)
such that
\begin{equation}
 Q_{h_j}^{e_j}u_j\ge cM_j
 \quad\text{on }B_{ch_j}(z_j).
\label{eq:directional-difference}
\end{equation}
The constants depend only on the data in \eqref{eq:blowup-assumptions} and on \(B\).
\end{lemma}
The main estimate in the proof of this lemma is a lower bound of order \(M_j\) for the Bessel resolvent average of \(\Delta u_j\) at scale $h_j$.
At that scale, the averaged trace is an average of
the second differences \(Q_yu_j\).  Hence one of these second differences
is of order \(M_j\).
The local Hessian bound implies that the same lower bound holds throughout a ball of radius comparable to \(h_j\).

We first record the properties of the Bessel resolvent \((I-s\Delta)^{-1}\) used in the proof. 

\begin{lemma}\label{lem:bessel-resolvent}
For \(s>0\), let
\[
\mathcal R_s=(I-s\Delta)^{-1}
\]
on \(\mathbb R^n\), and let \(G_s\) be its convolution kernel of total mass one.  Then $G_s>0$ and
\[
\mathcal R_s\phi(x)
=\int_{\mathbb R^n}G_s(x-y)\phi(y)\,dy.
\]
Moreover,
\[
G_s(y)
=\int_0^\infty
e^{-t}(4\pi st)^{-n/2}
\exp\left(-\frac{\abs{y}^2}{4st}\right)\,dt,
\]
and
\[
(G_s*G_s)(y)
=\int_0^\infty
te^{-t}(4\pi st)^{-n/2}
\exp\left(-\frac{\abs{y}^2}{4st}\right)\,dt.
\]

Let $r=\sqrt{s}$. There exist constants \(C,c>0\) and an integer \(N\), depending only on \(n\), such that for every \(d>0\),
\[
\int_{\abs y\ge d}G_s(y)\,dy
+
\int_{\abs y\ge d}(G_s*G_s)(y)\,dy
\le
C(1+d/r)^N e^{-cd/r}.
\]
We also have
\begin{equation}
\frac{(G_s*G_s)(y)}{G_s(y)}
\le C(1+\abs y/r).
\label{eq:bessel-kernel-ratio-bound}
\end{equation}

Finally, if \(w\in C_c^2(\mathbb R^n)\), then
\[
\mathcal R_{h^2}\Delta w(x)
=\int_{\mathbb R^n}Q_yw(x)\,dW_h(y),
\]
where
\[
Q_yw(x)
=\frac{w(x+y)+w(x-y)-2w(x)}{\abs y^2}
\]
and
\[
dW_h(y)
=\frac{\abs y^2}{2h^2}G_{h^2}(y)\,dy.
\]
The measure \(W_h\) has total mass \(n\).  Moreover, we have
\[
\lim_{\delta\downarrow0}
W_h({\abs y<\delta h})=0,
\qquad
\lim_{L\to\infty}
W_h({\abs y>Lh})=0,
\]
uniformly in \(h>0\).
\end{lemma}

\begin{proof}
For every \(\lambda\ge0\),
\[
\frac1{1+\lambda}=\int_0^\infty e^{-t}e^{-t\lambda}\,dt.
\]
Applying this identity to \(\lambda=-s\Delta\), we obtain
\[
\mathcal R_s
=\int_0^\infty e^{-t}e^{st\Delta}\,dt.
\]
The heat operator \(e^{st\Delta}\) has kernel
\[
(4\pi st)^{-n/2}
\exp\left(-\frac{\abs y^2}{4st}\right).
\]
This proves the formula for \(G_s\), as well as its positivity and total
mass.
Similarly \(G_s*G_s\), which is the kernel of \(\mathcal R_s^2\), has the
stated integral formula.

We prove the decay estimate.  The Gaussian estimate
\[
\int_{\abs y\ge d}
(4\pi r^2t)^{-n/2}
\exp\left(-\frac{\abs y^2}{4r^2t}\right)\,dy
\le
C\left(1+\frac{d}{r\sqrt t}\right)^N
\exp\left(-\frac{cd^2}{r^2t}\right)
\]
holds for every \(t>0\).  Put \(a=d/r\).  On \(0<t\le a\),
\(
t+\frac{ca^2}{t}\ge c_1a.
\)  Applying these inequalities to the
integral formulas for \(G_s\) and \(G_s*G_s\), we obtain
\[
\int_{\abs y\ge d}G_s(y)\,dy
+
\int_{\abs y\ge d}(G_s*G_s)(y)\,dy
\le C(1+a)^Ne^{-c_2a}.
\]
This is the required estimate.

Differentiating
\[
(I-s\Delta)\mathcal R_s=I
\]
with respect to \(s\), we obtain
\[
\partial_s\mathcal R_s
=\mathcal R_s\Delta\mathcal R_s.
\]
Since \(s\Delta=I-(I-s\Delta)\), it follows that
\[
s\partial_s\mathcal R_s
=\mathcal R_s^2-\mathcal R_s.
\]
For \(z>0\) and \(\mu\in\mathbb R\), let \(K_\mu(z)\) denote the
modified Bessel function of the second kind, defined by
\[
 K_\mu(z)
 =
 \frac12\left(\frac z2\right)^{-\mu}
 \int_0^\infty
 t^{\mu-1}
 \exp\left(-t-\frac{z^2}{4t}\right)\,dt.
\]
Let \(y\ne0\), set \(z=\abs y/r\).
Applying this identity with \(\mu=1-n/2\) to \(G_s\), and with
\(\mu=2-n/2\) to \(G_s*G_s\), gives
\begin{equation}
\frac{(G_s*G_s)(y)}{G_s(y)}
=\frac{z}{2}
\frac{K_{2-n/2}(z)}{K_{1-n/2}(z)},
\label{eq:bessel-kernel-ratio-formula}
\end{equation}
For \(0<z\le1\), we use
\[
K_0(z)\sim\abs{\log z},
\qquad
K_\nu(z)\sim
2^{\abs\nu-1}\Gamma(\abs\nu)z^{-\abs\nu}
\quad\text{if }\nu\ne0,
\]
together with \(K_{-\nu}=K_\nu\).  If \(n=2\), the right-hand side of
\eqref{eq:bessel-kernel-ratio-formula} is bounded by
\(C/\abs{\log z}\).  If \(n=3\), it equals \(z/2\).  If \(n=4\), it is
bounded by \(Cz^2\abs{\log z}\).  If \(n\ge5\), it is bounded by \(Cz^2\).
For \(z\ge1\),
\[
K_\nu(z)
=\left(\frac{\pi}{2z}\right)^{1/2}
e^{-z}\bigl(1+O(z^{-1})\bigr),
\]
so the quotient in \eqref{eq:bessel-kernel-ratio-formula} is bounded by \(Cz\).
This proves \eqref{eq:bessel-kernel-ratio-bound}.

It remains to prove the formula involving \(Q_yw\).  By symmetry of
\(G_{h^2}\),
\[
\begin{aligned}
\int_{\mathbb R^n}Q_yw(x)\,dW_h(y)
&=
\frac1{2h^2}
\int_{\mathbb R^n}
\bigl(w(x+y)+w(x-y)-2w(x)\bigr)G_{h^2}(y)\,dy\\
&=
\frac{\mathcal R_{h^2}w(x)-w(x)}{h^2}.
\end{aligned}
\]
Since
\(
\mathcal R_{h^2}-I
=h^2\mathcal R_{h^2}\Delta,
\)
the last expression equals
\(\mathcal R_{h^2}\Delta w(x)\).
By direct computation, \(W_h(\mathbb R^n)=n\).  Since
\(\abs z^2G_1(z)\) is integrable on \(\mathbb R^n\), its integral over
\(\abs z<\delta\) tends to zero as \(\delta\downarrow0\), and its
integral over \(\abs z>L\) tends to zero as \(L\to\infty\).
\end{proof}

\begin{proof}[Proof of \cref{lem:bessel}]
Subtract the affine tangent function of \(u_j\) at \(x_j\).
From \cref{lem:newton} and \eqref{eq:local-trace-cap},
\[
 \abs{D^2u_j}\le2M_j,\qquad
 \abs{Du_j}\le CM_jR_j,\qquad
 \abs{u_j}\le CM_jR_j^2
\]
on $B_{R_j}(x_j)$.
Choose a cutoff equal to one on \(B_{3R_j/4}(x_j)\), supported in \(B_{R_j}(x_j)\), and extend the product \(\widetilde u_j\) by zero.
Then
\[
 \abs{D^2\widetilde u_j}\le CM_j,\quad
 \abs{D\widetilde u_j}\le CM_jR_j,\quad
 \abs{\widetilde u_j}\le CM_jR_j^2.
\]

Set
\(
h_j^0=\frac1{M_jF(M_j)}.
\)
For \(s>0\), let \(\mathcal R_s\) and \(G_s\) be as in
\cref{lem:bessel-resolvent}, and define
\[
H_s=\mathcal R_sD^2\widetilde u_j,
\qquad
T_s=\tr H_s=\mathcal R_s\Delta\widetilde u_j.
\]
We shall consider \(s=r^2\) with \(0<r\le h_j^0\).
Let
\(
I_j=B_{3R_j/4}(x_j).
\)
For \(x\in B_{R_j/4}(x_j)\), define
\[
\beta_{r,x}
=\int_{I_j}G_{r^2}(x-y)\,dy.
\]
The decay estimate in \cref{lem:bessel-resolvent} gives
\[
0\le1-\beta_{r,x}
\le
\epsilon_r
:=
C(1+R_j/r)^N e^{-cR_j/r}.
\]
Note that 
\(
R_j/r
\ge R_j/h_j^0
=\ell_jM_j.
\)
Consequently,
\(
M_j^2\epsilon_r=o(1)
\)
uniformly for \(0<r\le h_j^0\).
Define a probability measure on \(I_j\) by
\[
 d\bar\nu_{r,x}(y)
 =
 \frac{\mathbf 1_{I_j}(y)G_{r^2}(x-y)}
      {\beta_{r,x}}\,dy,
\]
and set
\[
\bar H_{r,x}
=\int_{I_j}D^2u_j(y)\,d\bar\nu_{r,x}(y),
\qquad
\bar T_{r,x}
=\tr\bar H_{r,x}.
\]
On \(I_j\), we have \(\widetilde u_j=u_j\).  Therefore
\(
 H_{r^2}(x)
 =
 \beta_{r,x}\bar H_{r,x}+E_{r,x},
\)
where
\[
 E_{r,x}
 =
 \int_{\mathbb R^n\setminus I_j}
 D^2\widetilde u_j(y)G_{r^2}(x-y)\,dy.
\]
The bound for \(D^2\widetilde u_j\) implies
\[
\abs{E_{r,x}}\le CM_j\epsilon_r.
\]
Since \(\abs{\bar H_{r,x}}\le2M_j\), it follows that
\[
\abs{H_{r^2}(x)-\bar H_{r,x}}
\le CM_j\epsilon_r.
\]
Since \(G=\sqrt{\sigma_2}\) is concave on
\(\Gtwo\),
\[
 G(\bar H_{r,x})
 \ge
 \int_{I_j}G(D^2u_j(y))\,d\bar\nu_{r,x}(y)
 \ge\sqrt{f_0}.
\]
Thus \(\bar H_{r,x}\in\Gtwo\).
Moreover,
\[
\abs{
\sigma_2(H_{r^2}(x))
-\sigma_2(\bar H_{r,x})
}
\le CM_j^2\epsilon_r=o(1).
\]
Hence \(H_{r^2}(x)\in\Gtwo\) for all
\(x\in B_{R_j/4}(x_j)\), all \(0<r\le h_j^0\), and all sufficiently
large \(j\).

Choose \(\chi_0\in C_c^\infty(B_{R_j/4}(x_j))\) such that
\(0\le\chi_0\le1\), \(\chi_0(x_j)=1\),
and, wherever \(\chi_0>0\),
\[
\frac{\abs{D\chi_0}^2}{\chi_0}
+\abs{D^2\chi_0}
\le\frac{C}{R_j^2}.
\]
Define
\(
N(s)=\max_x\chi_0(x)T_s(x).
\)
Since \(\widetilde u_j=u_j\) near \(x_j\),
\(
N(0)\ge\Delta u_j(x_j)=M_j.
\)
Suppose that \(N(s)\ge M_j/2\).  Choose \(x_s\) such that
\(
N(s)=\chi_0(x_s)T_s(x_s).
\)
Since \(\abs{T_s}\le CM_j\), we have
\(
T_s(x_s)\sim M_j,
\)
and
\(\chi_0(x_s)\ge c.
\)
At \(x_s\),
\[
D(\chi_0T_s)=0,\qquad
D^2(\chi_0T_s)\le0.
\]
The first identity gives
\[
DT_s=-\frac{T_s}{\chi_0}D\chi_0.
\]
Substituting this into the second inequality, we obtain
\[
D^2T_s(x_s)
\le
C\frac{M_j}{R_j^2}\mathrm I.
\]

Let $\Slin(A):=D\sigma_2(A)=(\tr A)\mathrm{I}-A$.
Since
\(
\Slin(H_s+D^2u_j(x_s)
)>0,
\)
we have
\[
\begin{aligned}
\sigma_2(H_s)-f_j(x_s)
&=
\left\langle \Slin\left(
\frac{H_s+D^2u_j(x_s)}2
\right),
\bigl(H_s-D^2u_j(x_s)\bigr) \right\rangle \\
&\le
Cs\frac{M_j}{R_j^2}
\tr\Slin\left(
\frac{H_s+D^2u_j(x_s)}2
\right)\\
&\le C\frac{M_j^2r^2}{R_j^2}.
\end{aligned}
\]
We may enlarge the right-hand side by the positive quantity
\(CM_j^2\epsilon_r\).  Thus
\begin{equation}
\sigma_2(H_s)
\le
f_j(x_s)
+C\frac{M_j^2r^2}{R_j^2}
+CM_j^2\epsilon_r.
\label{eq:averaged-hessian-upper}
\end{equation}

We next compare the individual values of \(\Delta u_j\) with their
average.  For any probability measure \(\nu\) and any
\(\Gtwo\)-valued matrix function \(A\),
\begin{equation}
\sigma_2\left(\int A\,d\nu\right)
\ge
\tr\left(\int A\,d\nu\right)
\int\frac{\sigma_2(A)}{\tr A}\,d\nu.
\label{eq:sigma2-average}
\end{equation}
In the following formulas, write
\(
\bar H_r=\bar H_{r,x_s},
\bar T_r=\bar T_{r,x_s},
d\bar\nu_r=d\bar\nu_{r,x_s}.
\)
Let $g_j=\sqrt{f_j}$ and
\(
m=\int_{I_j}g_j\,d\bar\nu_r.
\)
Since
\[
\int_{\mathbb R^n}
\abs{x_s-y}^{\alpha}G_{r^2}(x_s-y)\,dy
=Cr^\alpha
\]
and \(\beta_{r,x_s}=1+o(1)\), the \(C^\alpha\) bound for \(g_j,f_j\)
implies
\begin{equation}
\abs{m-g_j(x_s)}\le Cr^\alpha,\qquad\abs{m^2-f_j(x_s)}\le Cr^\alpha.
\label{eq:right-side-average}
\end{equation}

Applying \eqref{eq:sigma2-average} to \(A=D^2u_j\) and
\(\nu=\bar\nu_r\), we obtain
\[
\bar T_r\int_{I_j}\frac{f_j}{\Delta u_j}\,d\bar\nu_r
\le\sigma_2(\bar H_r).
\]
On the other hand, Cauchy--Schwarz inequality gives
\[
\left(\int_{I_j}g_j\,d\bar\nu_r\right)^2
\le
\left(\int_{I_j}\Delta u_j\,d\bar\nu_r\right)
\left(\int_{I_j}\frac{f_j}{\Delta u_j}\,d\bar\nu_r\right).
\]
Thus
\[
0\le
\bar T_r\int_{I_j}\frac{f_j}{\Delta u_j}\,d\bar\nu_r-m^2.
\]
Since
\[
\abs{H_s-\bar H_r}\le CM_j\epsilon_r,
\]
we also have
\[
\sigma_2(\bar H_r)
\le\sigma_2(H_s)+CM_j^2\epsilon_r.
\]
Combining this inequality with \eqref{eq:sigma2-average} and \eqref{eq:right-side-average}, we obtain
\begin{equation}
0\le
\bar T_r\int_{I_j}\frac{f_j}{\Delta u_j}\,d\bar\nu_r
-m^2
\le
C\left(
r^\alpha+\frac{M_j^2r^2}{R_j^2}
+M_j^2\epsilon_r
\right).
\label{eq:trace-concentration-error}
\end{equation}
To convert this to trace concentration, set
\[
 d\nu=\frac{g_j}{m}\,d\bar\nu_r,\qquad
 X=\frac{m\Delta u_j}{g_j\bar T_r}.
\]
Then \(\int X\,d\nu=1\), and
\[
 \int\frac{(X-1)^2}{X}\,d\nu=\int X^{-1}\,d\nu-1.
\]
By \eqref{eq:trace-concentration-error} and Cauchy--Schwarz inequality,
\[
\int\abs{X-1}\,d\nu
\le
C\left(
r^{\alpha/2}
+\frac{M_jr}{R_j}
+M_j\epsilon_r^{1/2}
\right).
\]
With \eqref{eq:right-side-average}, we conclude that
\[
\int_{I_j}
\left|
\frac{\Delta u_j}{\bar T_r}-1
\right|d\bar\nu_r
\le
C\left(
r^{\alpha/2}
+\frac{M_jr}{R_j}
+M_j\epsilon_r^{1/2}
\right).
\]
Finally,
\[
\abs{T_s(x_s)-\bar T_r}\le CM_j\epsilon_r.
\]
The integral over \(\mathbb R^n\setminus I_j\) is also bounded by
\(C\epsilon_r\), because
\(\abs{\Delta\widetilde u_j}\le CM_j\) and
\(T_s(x_s)\sim M_j\).  We therefore obtain
\begin{equation}
\int_{\mathbb R^n}
\left|
\frac{\Delta\widetilde u_j(y)}{T_s(x_s)}-1
\right|
G_s(x_s-y)\,dy
\le Cd_r,
\label{eq:trace-flatness}
\end{equation}
where
\[
A(r)
:=
r^\alpha
+\frac{M_j^2r^2}{R_j^2}
+M_j^2\epsilon_r,
\qquad
d_r
:=
r^{\alpha/2}
+\frac{M_jr}{R_j}
+M_j\epsilon_r^{1/2}.
\]
In particular,
\[
A(r)\le d_r^2\le 3A(r).
\]

Define
\[
a_s(z):=
\frac{(G_s*G_s)(z)}{G_s(z)}.
\]
Then we have
\begin{equation}
\begin{aligned}
\frac{s\partial_sT_s(x_s)}{T_s(x_s)}
&=
\frac{\mathcal R_s^2\Delta\widetilde u_j(x_s)}
     {T_s(x_s)}
-1\\
&=
\int_{\mathbb R^n}
a_s(x_s-y)
\left(
\frac{\Delta\widetilde u_j(y)}{T_s(x_s)}-1
\right)
G_s(x_s-y)\,dy.
\end{aligned}
\label{eq:resolvent-log-derivative}
\end{equation}

The function \(N\) is locally Lipschitz on \((0,\infty)\).
At every point \(s\) where \(N\) is differentiable, the maximum
property of \(x_s\) gives
\[
\frac{d}{ds}N(s)
=
\chi_0(x_s)\partial_sT_s(x_s).
\]
Consequently,
\[
s\frac{d}{ds}\log N(s)
=
\frac{s\partial_sT_s(x_s)}{T_s(x_s)}
\]
for almost every such \(s\).

Fix \(L\ge2\).  On \(B_{Lr}(x_s)\), the estimate
\eqref{eq:bessel-kernel-ratio-bound} gives
\[
a_s(x_s-y)\le C(1+L).
\]
Therefore the contribution of \(B_{Lr}(x_s)\) is bounded from below by
\(
-C(1+L)d_r
\)
by \eqref{eq:trace-flatness}.

On
\(
I_j\setminus B_{Lr}(x_s),
\)
we have \(\Delta\widetilde u_j=\Delta u_j>0\).  Hence the negative part of
the integrand is bounded by
\[
\int_{\abs y\ge Lr}(G_s*G_s)(y)\,dy
\le C(1+L)^Ne^{-cL}.
\]
On \(\mathbb R^n\setminus I_j\), the distance from \(x_s\) is at least
\(R_j/2\), and
\[
\left|
\frac{\Delta\widetilde u_j}{T_s(x_s)}-1
\right|\le C.
\]
The decay estimate in \cref{lem:bessel-resolvent} bounds this contribution
by
\[
C(1+R_j/r)^Ne^{-cR_j/r}
\le C\epsilon_r.
\]

Choose
\(
L=C_0\bigl(1+\abs{\log A(r)}\bigr)
\)
with \(C_0\) sufficiently large.  Since \(A(r)\le d_r^2\) and
\(0<A(r)\le1\) for all sufficiently large \(j\), the two exponential
terms above are bounded by
\(
Cd_r\bigl(1+\abs{\log A(r)}\bigr).
\)
We conclude that
\begin{equation}
s\frac d{ds}\log N(s)
\ge
-Cd_r\bigl(1+\abs{\log A(r)}\bigr).
\label{eq:logN-differential}
\end{equation}

It remains to integrate \eqref{eq:logN-differential} for
\(0<r\le h_j^0\).  Since
\(
\abs{\log A(r)}\le\alpha\abs{\log r},
\)
we have
\begin{equation}
\int_0^{h_j^0}
r^{\alpha/2}
\bigl(1+\abs{\log A(r)}\bigr)\frac{dr}{r}
\le
C(h_j^0)^{\alpha/2}
\bigl(1+\abs{\log h_j^0}\bigr)
=o(1).
\label{eq:holder-error-integral}
\end{equation}

For the second term in the formula of $d_r$, set \(q=M_jr/R_j\).  Since
\(
0<q\le M_jh_j^0/R_j=1/\ell_j
\)
and \(A(r)\ge q^2\), we have
\(
\abs{\log A(r)}\le2\abs{\log q}.
\)
Hence
\begin{equation}
\int_0^{h_j^0}
\frac{M_jr}{R_j}
\bigl(1+\abs{\log A(r)}\bigr)\frac{dr}{r}
\le
C\int_0^{1/\ell_j}
(1+\abs{\log q})\,dq
=o(1).
\label{eq:scale-error-integral}
\end{equation}

For the last term, put \(t=R_j/r\).  Then
\(
t\ge\frac{R_j}{h_j^0}=\ell_jM_j.
\)
By the definition of \(\epsilon_r\),
\(
\epsilon_r^{1/2}
\le C(1+t)^Ne^{-ct}.
\)
Moreover, \(A(r)\ge r^\alpha=(R_j/t)^\alpha\), so
\[
\abs{\log A(r)}
\le C\bigl(1+\abs{\log R_j}+\log t\bigr).
\]
Therefore
\begin{equation}
\begin{aligned}
&\int_0^{h_j^0}
M_j\epsilon_r^{1/2}
\bigl(1+\abs{\log A(r)}\bigr)\frac{dr}{r}\\
&\quad\le
CM_j\int_{\ell_jM_j}^{\infty}
(1+t)^Ne^{-ct}
\bigl(1+\abs{\log R_j}+\log t\bigr)\frac{dt}{t}
=o(1).
\end{aligned}
\label{eq:tail-error-integral}
\end{equation}

By \eqref{eq:holder-error-integral}--\eqref{eq:tail-error-integral},
\[
\int_0^{h_j^0}
d_r\bigl(1+\abs{\log A(r)}\bigr)\frac{dr}{r}
=o(1).
\]
Suppose that \(N\) first reaches \(M_j/2\) at some
\(s_*\le(h_j^0)^2\).  Integrating
\eqref{eq:logN-differential} from \(\delta\) to \(s_*\), using
\(ds/s=2\,dr/r\), and then letting \(\delta\downarrow0\), we obtain
\[
N(s_*)\ge(1-o(1))N(0)\ge(1-o(1))M_j.
\]
This contradicts \(N(s_*)=M_j/2\).  Hence
\[
N((h_j^0)^2)\ge(1-o(1))M_j.
\]

Set
\(
s_j=(h_j^0)^2
\)
and choose \(z_j\) such that
\(
N(s_j)=\chi_0(z_j)T_{s_j}(z_j).
\)
Since \(\chi_0\) is supported in \(B_{R_j/4}(x_j)\), we have
\(
z_j\in B_{R_j/4}(x_j).
\)
Moreover, \(0<\chi_0(z_j)\le1\), and therefore
\[
T_{s_j}(z_j)
\ge N(s_j)
\ge(1-o(1))M_j.
\]
By the second-difference representation in \cref{lem:bessel-resolvent},
\[
\begin{aligned}
\int_{\mathbb R^n}
Q_y\widetilde u_j(z_j)\,dW_{h_j^0}(y)
&=
\mathcal R_{(h_j^0)^2}
\Delta\widetilde u_j(z_j)\\
&=
T_{(h_j^0)^2}(z_j)\\
&\ge(1-o(1))M_j.
\end{aligned}
\]

The global Hessian bound for \(\widetilde u_j\) implies
\[
\abs{Q_y\widetilde u_j(z_j)}\le CM_j.
\]
Choose \(\delta>0\) sufficiently small and \(L>1\) sufficiently large.
Since \(W_h(\mathbb R^n)=n\), the last part of \cref{lem:bessel-resolvent} and the
preceding bound imply that
\[
\int_{\abs y<\delta h_j^0}
\abs{Q_y\widetilde u_j(z_j)}\,dW_{h_j^0}(y)
+
\int_{\abs y>Lh_j^0}
\abs{Q_y\widetilde u_j(z_j)}\,dW_{h_j^0}(y)
\le\frac14M_j.
\]
Consequently,
\[
\int_{\delta h_j^0\le\abs y\le Lh_j^0}
Q_y\widetilde u_j(z_j)\,dW_{h_j^0}(y)
\ge cM_j.
\]
Since \(W_{h_j^0}(\mathbb R^n)=n\), there exists \(y_j\) such that
\[
\delta h_j^0\le\abs{y_j}\le Lh_j^0
\]
and
\[
Q_{y_j}\widetilde u_j(z_j)\ge cM_j.
\]
Set
\[
 h_j=\abs{y_j},\qquad e_j=\frac{y_j}{\abs{y_j}}.
\]
Then \(h_j\sim h_j^0=(M_jF(M_j))^{-1}\). 
Moreover,
\(h_j^0/R_j=(\ell_jM_j)^{-1}\to0\).  Since
\(z_j\in B_{R_j/4}(x_j)\) and \(\widetilde u_j=u_j\) on
\(B_{3R_j/4}(x_j)\), it follows for all sufficiently large \(j\) that
\[
 Q_{h_j}^{e_j}u_j(z_j)
 =Q_{y_j}\widetilde u_j(z_j)\ge cM_j.
\]

It remains to extend this inequality from \(z_j\) to a small ball.  For
\(x\) such that \(x\), \(x+h_je_j\), and \(x-h_je_j\) belong to
\(B_{R_j}(x_j)\),
\[
\abs{DQ_{h_j}^{e_j}u_j(x)}
= \frac1{h_j^2}
\abs{
Du_j(x+h_je_j)+Du_j(x-h_je_j)-2Du_j(x)
} \le\frac{CM_j}{h_j}.
\]
Consequently,
\[
Q_{h_j}^{e_j}u_j(x)
\ge
Q_{h_j}^{e_j}u_j(z_j)
-\frac{CM_j}{h_j}\abs{x-z_j}.
\]
After choosing \(c>0\) sufficiently small, we obtain
\[
Q_{h_j}^{e_j}u_j\ge cM_j
\quad\text{on }B_{ch_j}(z_j).
\]
This proves \eqref{eq:directional-difference} and finishes the proof of \cref{lem:bessel}.
\end{proof}

\subsection{A set where the trace is large}
\label{sec:large-trace-set}

Suppress the index \(j\).
Write \(M,R,h,\ell_M,x_M,z_M,e\) for \(M_j,R_j,h_j,\ell_j,x_j,z_j,e_j\), respectively, translate coordinates so that \(z_M=0\), and put \(q=\Delta u\).
The translated center domain is
\[
 \Omega_{\mathrm{cen}}^h
 =\{x:x+[-h,h]e\subset B_R(x_M)\}.
\]
All averages below are taken only in this domain.

Define
\[
 W_h(x)=u(x+he)+u(x-he)-2u(x),
\]
\[
 U_T(x)=\int_{-1}^1(1-\abs t)u(x+the)\,dt,\qquad Z_T=U_T-u.
\]
Choose a fixed sufficiently small \(\theta\in(0,1/4)\), and let
\[
 d\mu(t)=\frac{1-\theta}{2}(\delta_{-1}+\delta_1)
                +\theta(1-\abs t)\,dt,
\]
\[
 U_\mu(x)=\int u(x+the)\,d\mu(t),\qquad Z_\mu=U_\mu-u.
\]
By \(\abs{D^2u}<q\le2M\) on the stopping ball, we have
\[
 \abs{Z_\mu}\le CMh^2,\qquad
 \abs{DZ_\mu}\le CMh.
\]
Set \(a=c_aMh^2\), with \(c_a>0\) chosen as in \cref{lem:bessel}.
Since
\[
 \abs{Z_T}\le CMh^2\int_0^1(1-t)t^2\,dt,
\]
choosing \(\theta\) small gives
\begin{equation}
 Z_\mu\ge a/4
\quad\text{on a fixed smaller ball in }B_{ch}(z_M).
\label{eq:positive-average}
\end{equation}
Fix \(\ell_0=a/32\), let \(V\) be the relative component of \(\{Z_\mu>\ell_0\}\) containing the ball in \eqref{eq:positive-average}, and set
\[
 \widehat w=(Z_\mu-\ell_0)\mathbf 1_{V}.
\]
On every convex cylinder compactly contained in \(\Omega_{\mathrm{cen}}^h\), this zero extension is continuous and retains the \(CMh\) Lipschitz bound: a segment leaving \(V\) first meets an interior relative boundary on which \(Z_\mu=\ell_0\).
We make no assertion about a possible intersection of \(V\) with the outer boundary of \(\Omega_{\mathrm{cen}}^h\).

\begin{lemma}\label{lem:transverse-ellipticity}
Where \(\widehat w>0\), define
\[
 H_\mu=D^2U_\mu,\qquad
 K_\mu=\frac12(H_\mu+D^2u),\qquad
 B_\mu=\Slin(K_\mu).
\]
Then
\begin{equation}
 \langle B_\mu,D^2\widehat w\rangle\ge-Ch^\alpha,
\label{eq:transverse-equation}
\end{equation}
and, on this component,
\begin{equation}
 cMI_{e^\perp}\le B_\mu|_{e^\perp}\le CMI_{e^\perp},
\qquad 0<B_\mu\le CMI.
\label{eq:transverse-ellipticity}
\end{equation}
\end{lemma}

\begin{proof}
The cone is convex, so \(K_\mu\in\Gtwo\) and \(B_\mu>0\).
Quadratic polarization gives the exact equality
\[
 \langle B_\mu,D^2Z_\mu\rangle
 =\sigma_2(H_\mu)-\sigma_2(D^2u).
\]
Concavity and homogeneity of degree one give
\[
 G(H_\mu)\ge\int g(x+the)\,d\mu(t)\ge g(x)-Ch^\alpha.
\]
The fixed upper and lower bounds for \(g\) imply \eqref{eq:transverse-equation}.

The choice \(\ell_0=a/32\), which lies strictly above the triangular error, implies on the component
\[
 Q_h^eu=\frac{W_h}{h^2}\ge c_dM.
\]
For
\[
 H_T=\int_{-1}^1(1-\abs t)D^2u(x+the)\,dt,
\]
one has \(e^TH_Te=Q_h^eu\).
Since \(\Slin(H_T)>0\),
\[
 \tr H_T-e^TH_Te=e^T\Slin(H_T)e>0.
\]
Thus \(\tr H_T\ge c_dM\).
Because \(H_T\) occurs with fixed positive weight in \(K_\mu\), while all translated traces are at most \(2M\),
\begin{equation}
 cM\le q_K:=\tr K_\mu\le CM.
\label{eq:average-trace-range}
\end{equation}

We now prove the lower bound in \eqref{eq:transverse-ellipticity} without an eigenvector alignment assumption.
Write each matrix entering the convex average as \(H_i=q_iS_i\), where \(\tr S_i=1\), and write
\[
 S=K_\mu/q_K=\sum_i p_iS_i,\qquad p_i>0,\quad\sum_i p_i=1.
\]
The triangular component has \(p_T\ge p_*>0\) and
\begin{equation}
 e^TS_Te\ge c_*>0.
\label{eq:directional-mass}
\end{equation}
For an admissible matrix with trace one,
\[
 \abs{S_i}^2=1-2\sigma_2(S_i).
\]
The Euclidean variance identity therefore yields
\[
 \sum_i p_i\abs{S_i-S}^2
 =2\sigma_2(S)-2\sum_i p_i\sigma_2(S_i)
 \le2\sigma_2(S).
\]
If \(\sigma_2(S)\le\eta\), then
\[
 \abs{S_T-S}^2\le2\eta/p_*.
\]
For fixed small \(\eta\), \eqref{eq:directional-mass} gives \(e^TSe\ge c_e>0\).
Hence, for every unit \(\tau\perp e\),
\[
 1\ge\abs S^2\ge(\tau^TS\tau)^2+(e^TSe)^2,
\]
so \(\tau^TS\tau\le\sqrt{1-c_e^2}<1\).
It follows that
\[
 \tau^TB_\mu\tau=q_K(1-\tau^TS\tau)\ge cM.
\]
If \(\sigma_2(S)\ge\eta\), let \(\lambda_1\) be its largest eigenvalue and \(\nu\) a corresponding eigenvector.
The exact identity
\[
 1-\lambda_1
 =\sigma_2(S)+\frac12\abs{S-\nu\otimes\nu}^2
\]
shows \(1-\lambda_1\ge\eta\), and consequently \(\xi^TB_\mu\xi\ge\eta q_K\) for every unit \(\xi\).
This proves the lower bound.
The upper bound follows from \(\abs S<1\) and \(q_K\le CM\):
\[
 0<\xi^TB_\mu\xi=q_K(1-\xi^TS\xi)\le2q_K\le CM.
\]
\end{proof}

The transverse spreading argument uses the following local maximum principle, proved for example in \cite[Chapter~4]{CC1995}.
\begin{proposition}
\label{prop:local-maximum}
Let \(m\ge1\), \(0<\lambda\le\Lambda\), and let \(z\ge0\) be a continuous viscosity subsolution of
\[
 \mathcal M_{\lambda,\Lambda}^+(D^2z)\ge-\delta
 \quad\text{in }B_{2R}\subset\mathbb R^m.
\]
There exist \(p>0\) and \(C\), depending only on \(m,\lambda,\Lambda\), such that
\[
 \sup_{B_R}z
 \le C\left[
 \left(\frac{1}{\abs{B_{2R}}}\int_{B_{2R}}z^p\right)^{1/p}+\delta R^2\right].
\]
\end{proposition}
\begin{lemma}\label{lem:partial-sup}
There is a measurable set \(S\), contained in an \(O(R)\) neighborhood of \(z_M\), such that
\[
 \abs S\ge cR^{n-1}h,\qquad \widehat w\ge ca\quad\text{on }S.
\]
\end{lemma}

\begin{proof}
Choose fixed constants in the order
\begin{equation}
 T=c_TR,\qquad r_\perp=\varepsilon_1T,
\label{eq:cylinder-scales}
\end{equation}
first \(c_T>0\) small and then \(\varepsilon_1>0\) small.
Since \(\abs{z_M-x_M}\le R/4\) and \(h/R\to0\),
\[
 Q_{2r_\perp,T}:=(-T,T)e+B_{2r_\perp}^{e^\perp}
 \Subset\Omega_{\mathrm{cen}}^h
\]
for all large \(M\), including all of its \([-h,h]e\)-translations.

Fix \(A_*>\sup\widehat w/a+1\), and define on \(e^\perp\)
\[
 z(y)=\left[
 \max_{\abs t\le T}
 \left\{\widehat w(te+y)-A_*a\frac{t^2}{T^2}\right\}
 \right]_+.
\]
When \(z(y)>0\), every maximizing \(t\) lies strictly in \((-T,T)\); at the faces the penalty exceeds \(\sup\widehat w\).
If \(\phi\) touches \(z\) from above at \(y_0\), then
\[
 (t,y)\mapsto\phi(y)+A_*a\,t^2/T^2
\]
touches \(\widehat w\) from above at an interior maximizing point.
The mixed Hessian block of this lifted test is exactly zero.
By \cref{lem:transverse-ellipticity}, after division by \(M\),
\[
 \mathcal M_{\lambda,\Lambda}^+(D_y^2z)
 \ge-C\left(\frac{h^\alpha}{M}+\frac a{T^2}\right)
\quad\text{in }B_{2r_\perp}^{e^\perp}.
\]
At a point where \(z=0\), a test from above has a local minimum and hence nonnegative Hessian, so taking the positive part causes no extra term.

Now \(z(0)\ge ca\) and \(0\le z\le Ca\).
The local maximum principle \cref{prop:local-maximum} gives a set of fixed positive measure
\begin{equation}
 \abs{\{y:z(y)\ge c_1a\}}\ge cr_\perp^{\,n-1}
\label{eq:transverse-superlevel}
\end{equation}
provided
\[
 \frac{h^\alpha r_\perp^2}{M}\ll a,\qquad
 \frac{r_\perp}{T}\ll1.
\]
The second condition was imposed in \eqref{eq:cylinder-scales}.
Since \(a\sim Mh^2\), the first is
\[
 \left(\frac{r_\perp}{Mh^{1-\alpha/2}}\right)^2\ll1.
\]
Using \(r_\perp\sim R=\ell_M/F(M)\) and \(h\sim(MF(M))^{-1}\), its square root is at most
\[
 C\frac{\ell_M}{M^{\alpha/2}F(M)^{\alpha/2}}\longrightarrow0
\]
for every fixed \(\alpha>0\).

For each \(y\) in \eqref{eq:transverse-superlevel}, the penalized objective is at least \(c_1a\) at its maximizing point and is negative at the axial faces.
Its axial Lipschitz constant is at most \(CMh+Ca/T\le CMh\); hence the maximizing point is at least \(ch\) from either face that it approaches.
Since \(\widehat w\) itself is \(CMh\)-Lipschitz and \(a\sim Mh^2\), its fiber superlevel \(\{\widehat w\ge c_2a\}\) has length at least \(ch\).
This remains true near an interior component boundary because the zero extension has the same Lipschitz constant.
Applying Fubini directly to the continuous superlevel set, without selecting a measurable maximizing point \(t_y\), gives
\[
 \abs{\{\widehat w\ge c_2a\}\cap Q_{2r_\perp,T}}
 \ge cr_\perp^{n-1}h\ge cR^{n-1}h.
\]
\end{proof}

\begin{proof}[Proof of \cref{prop:large-trace-set}]
On the set \(S\) from \cref{lem:partial-sup}, the quantitative large value construction and \eqref{eq:average-trace-range} give a probability average
\[
 \int\Delta u(x+the)\,d\nu(t)\ge cM,
\]
where \(\nu\) is the fixed probability measure representing \(K_\mu=(D^2u+H_\mu)/2\).
Since every translated trace satisfies \(0<\Delta u\le2M\), elementary truncation yields fixed \(c',\theta_0>0\) such that
\[
 \nu\{t:\Delta u(x+the)\ge c'M\}\ge\theta_0
\quad\text{for every }x\in S.
\]
Let \(H_M=\{\Delta u\ge c'M\}\) and \(U=S+[-h,h]e\).
Then
\[
\begin{aligned}
 \theta_0\abs S
 &\le\int\!\int_S\mathbf 1_{H_M}(x+the)\,dx\,d\nu(t)\\
 &=\int\abs{H_M\cap(S+the)}\,d\nu(t)
 \le\abs{H_M\cap U}.
\end{aligned}
\]
Thus \(E=H_M\cap U\) has the required measure.
Translation overlap does not create a loss because each integrand in the last line is bounded by the measure of the same fixed union.
No measurable selection is used.

We state only the proved lower bound.
If a set with comparable upper and lower bounds is ever convenient, one may choose a measurable subset of prescribed measure because Lebesgue measure is nonatomic.
The main proof does not need that selection.
\end{proof}

\begin{remark}
Every gradient estimate in this section concerns the function after an affine function has been subtracted.
It follows by integrating the Hessian bound on the stopping ball.
No step uses \(\norm{Du_j}_{L^\infty(B_2)}\).
\end{remark}

\section{The comparison solution and its trace estimates}
\label{sec:comparison}

We now construct a smooth comparison solution and prove the two trace estimates used in the contradiction.
Auxiliary results are introduced when they are first needed.

\subsection{Construction of the comparison solution}
\label{sec:dirichlet-problem}
Fix once and for all a smooth function
\[
 \eta:[0,\infty)\longrightarrow[0,1]
\]
which is nondecreasing and concave, satisfies \(\eta(0)=0\), is identically one on \([\tau_*,\infty)\) for a fixed \(\tau_*>1\), and obeys
\[
 \eta(t)^2\le C_\eta t\quad(t\ge0).
\]
For example, one may integrate a smooth nonincreasing function that is positive on \([0,\tau_*)\) and vanishes to infinite order at \(\tau_*\), then normalize the integral.
Since \(\eta\) is Lipschitz and \(\eta(0)=0\), the displayed inequality follows after increasing \(C_\eta\).

The comparison solution solves the following modified Dirichlet problem.
The trace cutoff retains the original H\"older right side where the trace is moderate and replaces it by a smooth lower approximation where the trace is large.
\begin{theorem}
\label{thm:dirichlet-problem}
Let \(D=B_R(x_D)\), let \(u,g,b,A\) be smooth on \(\overline D\), and assume
\[
 D^2u\in\Gtwo,\qquad G(D^2u)=g,\qquad
 0<b\le g,\qquad A=\zeta(g-b),\quad
 \zeta\in C_c^\infty(D),\quad 0\le\zeta\le1.
\]
For every prescribed \(P>0\), the problem
\begin{equation}
 G(D^2v)+A(x)\eta\!\left(\frac{\Delta v}{P}\right)=g(x)
 \quad\hbox{in }D,
 \qquad v=u\quad\hbox{on }\partial D 
\label{eq:adaptive-dirichlet}
\end{equation}
has a unique smooth admissible solution.
Moreover, for every \(0<c_*<\inf_D b\), the problem
\[
 G(D^2w_0)=c_*,\qquad w_0=u\quad\hbox{on }\partial D,
\]
has a unique smooth admissible solution and
\begin{equation}
 u\le v\le w_0\quad\hbox{in }D.
\label{eq:solution-comparison}
\end{equation}
\end{theorem}

The proof uses the classical Dirichlet theorem and the concave estimates recorded next.
\begin{proposition}\label{prop:dirichlet}
Let \(D\) be a Euclidean ball, let \(h\in C^\infty(\overline D)\) be strictly positive, and let \(\phi\in C^\infty(\partial D)\).
If a smooth strictly \(2\)-admissible subsolution with boundary value \(\phi\) exists, then
\[
 \sigma_2(D^2w)=h\quad\text{in }D,\qquad w=\phi\quad\text{on }\partial D
\]
has a unique smooth \(2\)-admissible solution.
For a ball and strictly positive \(h\), the geometric hypotheses of the classical theorem are satisfied.
\end{proposition}

This is the \(k=2\) case of \cite[Theorem~3]{CNS1985}.
We use it only when the operator has already been fixed.
\begin{proposition}\label{prop:concave-estimates}
Let \(0<\gamma<1\), and let \(F:B_1\times\operatorname{Sym}(n)\to\mathbb R\) satisfy
\begin{align}
 \lambda\tr N
 &\le F(x,M+N)-F(x,M)\le\Lambda\tr N
 &&(N\ge0),
 \label{eq:concave-uniform-ellipticity}\\
 F(x,tM+(1-t)N)
 &\ge tF(x,M)+(1-t)F(x,N)
 &&(0\le t\le1),
 \label{eq:concave-operator}\\
 \abs{F(x,M)-F(y,M)}
 &\le K(1+\abs M)\abs{x-y}^{\gamma}.
 \label{eq:concave-x-holder}
\end{align}
If
\[
 F(x,D^2u)=f(x)\quad\hbox{in }B_1,
 \qquad
 \norm{u}_{C^2(B_1)}+\norm f_{C^\gamma(B_1)}\le K,
\]
then, for some \(\beta\in(0,\min\{\gamma,\bar\beta(n,\lambda,\Lambda)\})\),
\begin{equation}
 \norm u_{C^{2,\beta}(B_{1/2})}
 \le C(n,\lambda,\Lambda,\gamma,K).
\label{eq:concave-interior-estimate}
\end{equation}
The same estimate holds when \(F\) is initially defined only near the compact set \(\overline{D^2u(B_1)}\), provided it has a global concave extension satisfying \eqref{eq:concave-uniform-ellipticity} and \eqref{eq:concave-x-holder}.

For the boundary version, let \(\Omega\) and \(u|_{\partial\Omega}=\phi\) be smooth, and set
\[
 \Omega_\rho=\{x\in\Omega:
                 \dist(x,\partial\Omega)<\rho\}.
\]
If the equation on \(\Omega_\rho\) is a fixed smooth concave uniformly elliptic equation and \(\norm u_{C^2(\overline\Omega)}\le K\), then
\begin{equation}
 \norm u_{C^{2,\beta}(\overline{\Omega_{\rho/2}})}\le C,
\label{eq:concave-boundary-estimate}
\end{equation}
where \(C\) depends only on the fixed equation, \(K\), \(\Omega\), \(\rho\), and the corresponding norms of the data.
\end{proposition}

The interior estimate is the Evans--Krylov estimate \cite{Evans1982,Krylov1982}.
The boundary estimate is the version used in \cite[Section~4, pp.~277--279]{CNS1985}.
We use both only after obtaining a uniform Hessian bound along the continuity path.

\begin{proof}[Proof of \cref{thm:dirichlet-problem}]
We give the full continuity argument because it shows that the prescribed number \(P\) is not restricted by a local implicit-function radius.

Fix \(0<c_*<\inf_D b\).
Since \(G(D^2u)=g>c_*\), the function \(u\) is a strict admissible subsolution for the constant equation.
Thus \cref{prop:dirichlet}, applied with right side \(c_*^2\), gives the stated solution \(w_0\).

We begin with the continuity path and the zeroth-order comparison.
For \(0\le t\le1\), define
\[
 \mathcal F_t(x,H)=G(H)+tA(x)\eta(\tr H/P)
\]
and solve \(\mathcal F_t(x,D^2v_t)=g\) with \(v_t=u\) on \(\partial D\).
The path starts at \(v_0=u\).
Since \(G\) is concave, and \(\eta\) is concave and nondecreasing, \(H\mapsto\mathcal F_t(x,H)\) is concave.
Its linearization is
\[
 \mathcal L_t=
 \left(G^{ij}(D^2v_t)+\frac{tA}{P}
       \eta'(\Delta v_t/P)\delta_{ij}\right)\partial_{ij},
\]
which is elliptic on \(\Gtwo\).

At \(u\), the left side of the path equation is at least \(g\).
At \(w_0\), it is at most
\[
 c_*+tA\le c_*+g-b<g.
\]
Comparison, with the convention that the function with the larger operator value lies below when the boundary values agree, gives \(u\le v_t\le w_0\).
Consequently
\begin{equation}
 \norm{v_t}_{L^\infty(D)}+\norm{w_0}_{L^\infty(D)}
 \le C\left(1+\norm{u}_{L^\infty(D)}\right).
\label{eq:linfty-bound}
\end{equation}
This is the only continuity-path bound used later in the contradiction.

We next construct a strict comparison function and bound the gradient.
Choose \(\kappa>0\), depending on the fixed smooth problem, so large that
\[
 \underline u=u+\kappa(\abs{x-x_D}^2-R^2)
\]
is strictly admissible and
\[
 \mathcal F_t(x,D^2\underline u)\ge g+1
 \quad\hbox{for every }0\le t\le1.
\]
The boundary trace of \(\underline u\) is again \(u\).
Concavity at \(D^2v_t\) yields
\[
 \mathcal L_t(\underline u-v_t)\ge1.
\]
Differentiating the path equation in a coordinate direction gives the exact identity
\[
 \mathcal L_t(v_t)_k=g_k-tA_k\eta(\Delta v_t/P).
\]
In particular, its right side has no third derivative and no uncontrolled Hessian factor.
Apply the maximum principle to \(\pm(v_t)_k+C(\underline u-v_t)\).
This bounds the interior gradient by the boundary gradient.
On \(\partial D\), the zeroth-order comparison and the common boundary trace give, for the outward normal,
\[
 (w_0)_\nu\le(v_t)_\nu\le u_\nu.
\]
The tangential derivatives are fixed by the boundary data.
Hence
\[
 \sup_{0\le t\le1}\norm{Dv_t}_{L^\infty(D)}<\infty.
\]
This constant may depend on all derivatives of the fixed smooth data and on \(P^{-1}\); it is used only for existence of this one Dirichlet problem.

We now bound all boundary second derivatives.
Because \(A\) is compactly supported, there is a fixed open collar of \(\partial D\) in which every path member satisfies
\[
 G(D^2v_t)=g.
\]
Twice differentiating the common boundary trace tangentially, and using the gradient bound, controls all tangential--tangential derivatives.

We include the mixed-derivative barrier.
Let \(T\) be an infinitesimal rotation tangent to the sphere and put
\[
 \phi=T(v_t-u),\qquad W=v_t-\underline u.
\]
In the fixed-equation collar, the orthogonal invariance of \(G\) gives
\[
 G^{ij}(D^2v_t)(Tv_t)_{ij}=Tg.
\]
Indeed the commutator terms obtained by differentiating the rotation cancel the corresponding terms in \(D^2(Tv_t)\).
Write \(\mathcal L=G^{ij}(D^2v_t)\partial_{ij}\) there.
Concavity and the equation for \(u\) give \(\mathcal L(u-v_t)\ge0\), and therefore
\[
 \mathcal LW\le-2\kappa\sum_iG^{ii}.
\]
The positive density window and Newton's inequality imply \(1\le C\sum_iG^{ii}\), while the fixed smooth data give
\[
 \abs{\mathcal L\phi}\le C\left(1+\sum_iG^{ii}\right).
\]
On the outer edge of the collar, \(\phi=W=0\).
Choose the artificial inner edge strictly inside the collar.
There the gradient bound controls \(\phi\), whereas
\[
 W=v_t-u+\kappa(R^2-\abs{x-x_D}^2)
\]
has a fixed positive lower bound.
Increasing \(C\), the functions \(\pm\phi-CW\) are nonpositive on both edges and have nonnegative \(\mathcal L\)-image.
The maximum principle gives
\[
 \abs{T(v_t-u)}\le C(v_t-\underline u)
\]
throughout the collar.
Divide by the inward distance and approach the outer boundary.
This bounds the tangential--normal derivatives; derivatives of the rotation coefficients contain only first derivatives already bounded.

It remains to control the double normal derivative.
In an orthonormal tangential--normal frame, write
\[
 D^2v_t=\begin{pmatrix}B_t&p_t\\p_t^T&r_t\end{pmatrix}.
\]
The function \(q=v_t-w_0\) is nonpositive in \(D\), vanishes on the boundary, and hence satisfies \(q_\nu\ge0\) for the outward normal.
Since the second fundamental form of the sphere is positive, twice tangentially differentiating \(q=0\) gives
\[
 D^2_{\tan}q=q_\nu\,\mathrm{II}\ge0.
\]
Thus \(B_t\ge B_0=D^2_{\tan}w_0\), and
\[
 \tr B_t\ge\tr B_0
 =\Slin(D^2w_0)(\nu,\nu)\ge c_0>0.
\]
Here \(c_0\) may depend on the fixed solution \(w_0\).
The block identity
\[
 \sigma_2\!\begin{pmatrix}B&p\\p^T&r\end{pmatrix}
 =\sigma_2(B)+r\tr B-\abs{p}^2
\]
and the fixed equation in the collar give
\[
 r_t=\frac{g^2-\sigma_2(B_t)+\abs{p_t}^2}{\tr B_t}.
\]
Every term on the right is now bounded, so the path has a uniform boundary \(C^2\) bound.

We next obtain the interior Hessian bound.
For a unit vector \(\xi\), consider
\[
 Q_\xi=(v_t)_{\xi\xi}+C(\underline u-v_t).
\]
At an interior maximum where \(\Delta v_t<\tau_*P\), the cone inequality \(\lambda_{\max}<\Delta v_t\) directly bounds \((v_t)_{\xi\xi}\).
If \(\Delta v_t\ge\tau_*P\), then the trace cutoff is on its constant plateau and, through second order,
\[
 G(D^2v_t)=h_t:=g-tA.
\]
This remains true at the endpoint because all derivatives of \(\eta\) vanish there.
Twice differentiating and using the concavity of \(G\) gives
\[
 \mathcal L_t(v_t)_{\xi\xi}
 =(h_t)_{\xi\xi}
 -G^{ij,rs}(v_t)_{ij\xi}(v_t)_{rs\xi}
 \ge-\norm{D^2h_t}_{L^\infty(D)}.
\]
After \(C\) is chosen large enough, the strict inequality \(\mathcal L_t(\underline u-v_t)\ge1\) contradicts \(\mathcal L_tQ_\xi\le0\) at a large interior maximum.
Together with the boundary \(C^2\) bound, this proves
\[
 \sup_{0\le t\le1}\norm{D^2v_t}_{L^\infty(D)}<\infty.
\]
The argument first bounds all directional second derivatives from above and therefore bounds the trace.
The lower eigenvalue bound in \eqref{eq:eigenvalue-bounds} then bounds the Hessian from below.
The transition region is never differentiated: there the direct trace bound applies.

We finish the continuity argument by proving openness, closedness, and uniqueness.
The equation gives
\[
 G(D^2v_t)=g-tA\eta(\Delta v_t/P)\ge g-A\ge b>0.
\]
The Hessians therefore stay in a compact subset of \(\Gtwo\), and the linearized operators are uniformly elliptic.
Linear Dirichlet Schauder theory gives openness.

For closedness, choose a compact matrix set containing the Hessians of all path members and extend \(\mathcal F_t(x,\cdot)\) by the infimum of its supporting planes over that set.
The extension is globally concave and uniformly elliptic, with uniformly H\"older \(x\)-dependence.
Interior Evans--Krylov estimates apply away from the boundary.
In a smaller outer half-collar, where \(A=0\), the fixed-equation boundary \(C^{2,\beta}\) estimate applies using the boundary \(C^2\) bound already proved.
Interior balls overlap the inner edge of that collar, so these estimates cover \(\overline D\).
Higher regularity follows by linear Schauder bootstrapping.
Thus the path is closed and reaches \(t=1\).

If \(v\) and \(\widehat v\) are two solutions, integrate the Hessian linearization of \(\mathcal F_1\) along the segment joining their Hessians.
The cone \(\Gtwo\) is convex, so the resulting coefficient matrix is positive definite.
The maximum principle applied to \(v-\widehat v\), which vanishes on the boundary, proves uniqueness.
\end{proof}

The smooth lower approximation used in the construction is chosen as follows.
\begin{lemma}\label{lem:lower-mollification}
Let \(V\Subset V'\) be bounded domains, let \(0<\alpha<1\), and suppose
\[
 g\in C^\alpha(\overline {V'}),\qquad 0<g_0\le g\le g_1.
\]
For all sufficiently small \(\eps>0\), there is \(b_\eps\in C^\infty(\overline V)\) such that
\begin{equation}
 0<\frac{g_0}{2}\le b_\eps\le g\quad\hbox{on }V,
 \qquad
 \delta_\eps:=\norm{g-b_\eps}_{L^\infty(V)}\le C\eps^\alpha,
\label{eq:mollification-error}
\end{equation}
and
\begin{equation}
 L_\eps:=\norm{D(b_\eps^2)}_{L^\infty(V)}
       \le C\eps^{\alpha-1},
 \qquad
 J_\eps:=\norm{D^2(b_\eps^2)}_{L^\infty(V)}
       \le C\eps^{\alpha-2}.
\label{eq:mollification-derivatives}
\end{equation}
Here \(C\) depends only on the displayed bounds for \(g\), the H\"older seminorm of \(g\), and \(V\Subset V'\).
\end{lemma}

\begin{proof}
Extend \(g\) from \(V'\) to a \(C^\alpha\) function on a neighborhood of \(\overline V\), and let \(\rho_\eps\) be a standard nonnegative mollifier.
For a fixed sufficiently large \(C_0\), set \(b_\eps=\rho_\eps*g-C_0\eps^\alpha\).
The standard convolution estimates give \eqref{eq:mollification-error}.
They also give \(\abs{D^k b_\eps}\le C\eps^{\alpha-k}\) for \(k=1,2\).
The product rule and \(2\alpha-2>\alpha-2\) then give \eqref{eq:mollification-derivatives}.
\end{proof}
Suppose a sequence with unbounded trace has produced points \(z_j\to z_*\) and sets \(E_j\subset B_{CR_j}(z_j)\), where \(R_j\to0\).
Choose, in this order, fixed concentric balls
\begin{equation}
 U\Subset D_{\rm cmp}\Subset D_{\rm flat}\Subset D\Subset B_2,
\label{eq:nested-domains}
\end{equation}
centered at \(z_*\), and then choose
\begin{equation}
 \zeta\in C_c^\infty(D),\qquad 0\le\zeta\le1,\qquad
 \zeta=1\quad\hbox{on a neighborhood of }\overline {D_{\rm flat}}.
\label{eq:fixed-cutoff}
\end{equation}
For all large \(j\), \(E_j\Subset U\).
Notice the order of the quantifiers: the cutoff has been fixed before any modified solution and therefore before the comparison component, which will be contained in \(D_{\rm cmp}\).

For each fixed choice of \(\eps\) and \(P\), \cref{thm:dirichlet-problem} now supplies the required smooth comparison solution.

\subsection{A lower trace bound}

The next proposition transfers the lower trace bound from the original solution to the comparison solution.

\begin{proposition}\label{prop:comparison-lower-trace}
Let \(u_j,M_j,E_j\) be as in \cref{prop:large-trace-set}. Define \(S_j:=M_jR_j^{n-1}h_j\).
Let \(b_j\le g_j=\sqrt{f_j}\), set
\[
 \delta_j=\norm{g_j-b_j}_{L^\infty(D)},
\]
and let \(v_j\) solve \eqref{eq:adaptive-dirichlet} with parameter \(P_j\).
Assume that \(\norm{v_j}_{L^\infty(D)}\) is uniformly bounded and
\[
 \frac{M_j\delta_j^2}{P_jS_j}\longrightarrow0.
\]
Then, for all sufficiently large \(j\), there is a measurable set \(E'_j\subset E_j\) such that
\[
 \abs{E'_j}\ge\frac12\abs{E_j},\qquad
 \Delta v_j\ge cM_j\quad\hbox{on }E'_j.
\]
\end{proposition}

The transfer uses the quadratic polarization of \(\sigma_2\).
For \(A,B\in\operatorname{Sym}(n)\), define
\begin{equation}
 \ip AB:=\frac12\big((\tr A)(\tr B)-\langle A,B\rangle\big).
\label{eq:polarization}
\end{equation}
Thus \(\ip AA=\sigma_2(A)\).
The following is the corresponding reverse Cauchy--Schwarz inequality on \(\Gamma_2\).

\begin{lemma}\label{lem:reverse-cs}
If \(A,B\in\Gtwo\), then
\begin{equation}
 \ip AB\ge G(A)G(B).
\label{eq:reverse-cs}
\end{equation}
Consequently
\[
 \dd(A,B)^2:=2\big(\ip AB-G(A)G(B)\big)\ge0.
\]
Fix \(0<m_0\le m_1<\infty\).
There exist \(\kappa,c_L>0\), depending only on \(n,m_0,m_1\), such that if
\[
 m_0\le G(A),G(B)\le m_1,\qquad
 \tr A\ge M,\qquad \tr B\le\kappa M,
\]
then
\begin{equation}
 \dd(A,B)^2\ge c_L.
\label{eq:trace-gap}
\end{equation}
\end{lemma}

\begin{proof}
Write
\[
 A=T\left(\frac In+\sqrt{\frac{n-1}{n}}\,\Theta\right), \quad B= S \left(\frac In+\sqrt{\frac{n-1}{n}}\,\Xi\right),
 \quad \tr\Theta = \tr \Xi =0,
\]
with $T = \tr A, S = \tr B$.
We calculate
\[
 \sigma_2(A)=\frac{n-1}{2n}T^2(1-\abs{\Theta}^2).
\]
Since \(A, B\in\Gtwo\), we also have \(\abs\Theta<1, \abs\Xi<1\) and
\[
 \ip AB=\frac{n-1}{2n}TS\big(1-\langle\Theta,\Xi\rangle\big).
\]
Since
\[
 1-\langle\Theta,\Xi\rangle\ge 1-\abs\Theta\abs\Xi
 \ge\sqrt{(1-\abs\Theta^2)(1-\abs\Xi^2)},
\]
we obtain \eqref{eq:reverse-cs}.

For the last assertion, write $\widehat A = A / G(A),\, \widehat B = B / G(B)$, and set \(r=\operatorname{arctanh}\abs\Theta, s=\operatorname{arctanh}\abs\Xi\).
The displayed formulas imply
\[
 \ip{\widehat A}{\widehat B}\ge\cosh(r-s),
 \qquad
 \frac{\tr\widehat A}{\tr\widehat B}
 =\frac{\cosh r}{\cosh s}\le e^{\abs{r-s}}.
\]
Thus an upper bound for \(\dd(A,B)\) bounds the ratio of the normalized traces.
Since \(m_0\le G(A),G(B)\le m_1\), a sufficiently small \(\kappa\) implies \eqref{eq:trace-gap}.
\end{proof}

We also need its pointwise and integrated quadratic consequences.
\begin{lemma}\label{lem:point-square}
For \(A,B\in\Gtwo\),
\[
 \dd(A,B)^2
 =\big(G(A)-G(B)\big)^2-\sigma_2(A-B).
\]
\end{lemma}

\begin{proof}
Use the quadratic polarization
\[
 \sigma_2(A-B)=\sigma_2(A)+\sigma_2(B)-2\ip AB
\]
and expand.
\end{proof}
\begin{proposition}\label{prop:square}
Let \(D\) be a \(C^2\) domain and let \(u,v\in C^3(\overline D)\) be admissible with \(u=v\) on \(\partial D\).
Set \(w=u-v\), and let \(H_{\partial D}\) be the sum of the principal curvatures with respect to the outward normal.
Then
\[
 \;
 \int_D\dd(D^2u,D^2v)^2
 +\frac12\int_{\partial D}H_{\partial D}(x)
       \big[\partial_\nu w(x)\big]^2\,dS(x)
 =\int_D\big(G(D^2u)-G(D^2v)\big)^2.\;
\]
If \(D\) is mean convex, the boundary term is nonnegative.
\end{proposition}

\begin{proof}
By \cref{lem:point-square}, it remains to evaluate \(\int_D\sigma_2(D^2w)\).
Since \(\Slin(D^2w)\) is divergence free and \(w=0\) on \(\partial D\), integration by parts gives
\[
 2\int_D\sigma_2(D^2w)
 =\int_{\partial D}H_{\partial D}(x)
       \big[\partial_\nu w(x)\big]^2\,dS(x).
\]
This proves the identity.
\end{proof}
\begin{lemma}\label{lem:comparison-l2}
Let \(u\) and \(v\) have the same boundary values on the ball \(D\), with
\[
 G(D^2u)=g,\qquad
 G(D^2v)+\zeta(g-b)\eta(\Delta v/P)=g.
\]
Assume \(u,v\) are admissible, \(b\le g\), and \(\norm v_{L^\infty(D)}\le A_0\).
Then
\begin{equation}
 \int_D\dd(D^2u,D^2v)^2
 \le C\frac{\norm{g-b}_{L^\infty(D)}^2}{P},
\label{eq:comparison-integral}
\end{equation}
where \(C\) depends only on \(A_0\), \(\eta\), \(D\), and the fixed cutoff \(\zeta\).
\end{lemma}

\begin{proof}
The ball is mean convex.
By \cref{prop:square} and the equation for \(v\),
\[
 \begin{aligned}
 \int_D\dd(D^2u,D^2v)^2
 &\le\int_D\zeta^2(g-b)^2\eta(\Delta v/P)^2\\
 &\le C\frac{\norm{g-b}_\infty^2}{P}
       \int_{\supp\zeta}\Delta v.
 \end{aligned}
\]
Choose a fixed nonnegative \(\chi\in C_c^\infty(D)\) which equals one on \(\supp\zeta\).
Admissibility gives \(\Delta v>0\), and two integrations by parts, with no boundary term, yield
\[
 \int_{\supp\zeta}\Delta v
 \le\int_D\chi\Delta v
 =\int_Dv\Delta\chi
 \le A_0\int_D\abs{\Delta\chi}.
\]
Substitution proves \eqref{eq:comparison-integral}.
In the application, the required \(L^\infty\) bound for \(v\) is \eqref{eq:linfty-bound}; no derivative estimate from the continuity method has entered.
\end{proof}

\begin{proof}[Proof of \cref{prop:comparison-lower-trace}]
By \cref{lem:comparison-l2},
\[
 \int_D\dd(D^2u_j,D^2v_j)^2\le C\frac{\delta_j^2}{P_j}.
\]
The values of \(G(D^2u_j)\) and \(G(D^2v_j)\) lie in a fixed positive interval.
Hence \eqref{eq:trace-gap} gives a fixed \(\kappa>0\) such that
\[
 \abs{E_j\cap\{\Delta v_j<\kappa cM_j\}}
 \le C\frac{\delta_j^2}{P_j}.
\]
Since \(\abs{E_j}\ge cS_j/M_j\), the last measure is \(o(\abs{E_j})\).
Removing this set from \(E_j\) proves the proposition.
\end{proof}

\subsection{An upper trace bound}

We next prove the complementary upper estimate.
The comparison component, the gradient estimate, and the Pogorelov-type calculation are kept as separate steps below.

\begin{proposition}\label{prop:comparison-upper-trace}
Let \(U\Subset D_{\rm cmp}\Subset D_{\rm flat}\Subset\{\zeta=1\}\) be the fixed domains above.
Suppose that \(v\) solves \eqref{eq:adaptive-dirichlet}, that \(\norm v_{L^\infty}\) is bounded, and that
\[
 \Delta v\ge cM\quad\hbox{at some point of }U,\qquad P=o(M).
\]
Put
\[
 L=\norm{D(b^2)}_\infty,\qquad J=\norm{D^2(b^2)}_\infty.
\]
There exist a smooth admissible function \(W\) and a smooth component \(U\Subset\Omega\Subset D_{\rm cmp}\) such that, with
\[
 K=1+\norm{Dv}_{L^\infty(\Omega)}
      +\norm{DW}_{L^\infty(\Omega)},
\]
one has
\[
 K\le C\left(1+P^{1/2}+L^{1/5}\right)
\]
and
\[
 \sup_U\Delta v
 \le C\left(1+K^2+KL+K\sqrt J\right).
\]
All constants depend only on the fixed data and domains.
\end{proposition}

We first construct the comparison component, then control the gradient on the whole component, and finally apply the interior Hessian calculation.

\begin{proposition}\label{prop:modulus}
Let \(B_r\Subset B_R\) and let \(v\in C^\infty(B_R)\) be \(2\)-admissible.
If
\[
 \norm{v}_{L^\infty(B_R)}\le A,
 \qquad 0\le \sigma_2(D^2v)\le\Lambda,
\]
then there exist \(\beta\in(0,1)\) and \(C=C(n,A,\Lambda,r,R)\) such that
\begin{equation}
 [v]_{C^\beta(B_r)}
 =\sup_{\substack{x,y\in B_r\\x\ne y}}
   \frac{\abs{v(x)-v(y)}}{\abs{x-y}^\beta}
 \le C.
\label{eq:admissible-modulus}
\end{equation}
\end{proposition}

For \(n\ge4\), this follows from \cite[Corollary~3.4]{Labutin2002} with \(k=2\) and \(\mu_2[v](B_s)\le Cs^n\).
For \(n=3\), use \cite[Theorem~4.1]{TW1999} with \(3<p<6\), followed by Morrey's embedding.
For \(n=2\), use the interior slope bound for bounded convex functions.

\begin{proposition}\label{prop:gap}
Let \(v_0\) be a smooth admissible solution with constant data in a ball, and let \(v\) be admissible with a strictly larger right side.
Suppose
\[
 q=v_0-v\ge0,\qquad
 L_{v_0}q\le-c_0<0,
\]
where \(L_{v_0}=\Slin(D^2v_0)^{ij}\partial_{ij}\).
If \(q\ge\delta\) on a ball of fixed radius and \(\norm{Dv_0}\) is bounded on a fixed larger ball, then a finite chain of overlapping balls propagates a positive fraction of this gap to any prescribed compact set.
The constant depends on the gradient bound for \(v_0\), the balls, and \(c_0\), but not on \(Dv\) or derivatives of the right side of the equation for \(v\).
\end{proposition}

\begin{proof}
We spell out the local barrier in the present notation.
A fixed dilation and multiplication of the dependent variable normalize the constant data and the relevant balls as in \cite[Lemmas~3.2--3.4]{LiWu2026}.
Put
\[
 {\cal L}=\Slin(D^2v_0)^{ij}\partial_{ij}.
\]
The propagation step uses only
\[
 \sigma_2(D^2v_0)=\hbox{constant},\qquad q\ge0,\qquad
 {\cal L}q\le-c_0,
\]
together with the interior bound for \(Dv_0\).

Fix a sufficiently small radius \(r\), depending only on the normalized geometry and that gradient bound.
For a point \(y\) in the target compact set, define
\[
 \phi_y(x)=2\big((x-y)\cdot Dv_0(x)-v_0(x)+v_0(y)\big)
       +\frac{\alpha_0}{2}\abs{x-y}^2-2\beta_0\abs{Dv_0(x)}^2.
\]
Choose, in order, \(\beta_0>0\) small, \(\alpha_0\) large, and \(\gamma_0>0\) small.
The calculation in \cite[Lemma~3.2]{LiWu2026}, which differentiates only \(v_0\), produces a regular value \(c_y\), a smooth domain
\[
 B_{2r}(y)\subset\Omega_y\Subset B_{4r}(y),
\]
and the barrier
\[
 w_y=\exp\big((c_y-\phi_y)/\gamma_0\big)-1
\]
with
\[
 w_y=0\quad\hbox{on }\partial\Omega_y,\qquad
 0<w_y\le C_0\quad\hbox{in }\Omega_y,\qquad
 w_y\ge c_1\quad\hbox{on }B_{2r}(y),
\]
and
\[
 {\cal L}w_y>0\quad\hbox{in }\Omega_y\setminus B_r(y).
\]
The three cases in the verification of the last inequality use the Newton inequalities and the constant equation for \(v_0\).
Neither the equation nor the gradient of \(v\) is differentiated.

Suppose \(q\ge\delta\) on \(B_r(y)\).
On the boundary of the component of \(\Omega_y\setminus B_r(y)\) which meets \(B_{2r}(y)\setminus B_r(y)\),
\[
 q\ge\frac{\delta}{C_0}w_y.
\]
Indeed this holds on the inner boundary by \(w_y\le C_0\), and on the outer boundary because \(w_y=0\) and \(q\ge0\).
Since \({\cal L}(q-\delta w_y/C_0)<0\), the minimum principle gives the same inequality throughout that component.
Consequently
\[
 q\ge\theta\delta\quad\hbox{on }B_{2r}(y),
 \qquad \theta=\min\{1,c_1/C_0\}>0.
\]
A finite chain of overlapping \(r\)-balls propagates this estimate from the seed ball to the prescribed compact set.
Every constant depends only on the quantities listed in the proposition.
\end{proof}

\begin{lemma}\label{lem:component}
Fix concentric nested balls
\[
 U\Subset D_{\mathrm{cmp}}\Subset D_{\mathrm{out}}.
\]
Suppose \(v\in C^\infty(D_{\mathrm{out}})\) is admissible and
\[
 0<\lambda\le G(D^2v)\le\Lambda,\qquad
 \norm{v}_{L^\infty(D_{\mathrm{out}})}\le A.
\]
There exist a smooth admissible \(W\), a smooth connected component \(\Omega\), and constants depending only on \(n,\lambda,\Lambda,A\) and the fixed domain gaps such that
\begin{equation}
 U\Subset\Omega\Subset D_{\mathrm{cmp}},
\label{eq:component-inclusions}
\end{equation}
\[
 \rho:=W-v>0\quad\text{in }\Omega,\qquad
 \rho=0\quad\text{on }\partial\Omega,
\]
\begin{equation}
 \rho\ge c_\rho>0\quad\text{on }U,\qquad
 0<\rho\le C_\rho,\qquad
 \norm{DW}_{L^\infty(\Omega)}\le C.
\label{eq:component-bounds}
\end{equation}
These constants do not depend on derivatives of the right side in the equation for \(v\).
\end{lemma}

\begin{proof}
The proof for arbitrary fixed concentric radii is identical after changing fixed numerical constants.
To avoid carrying three radii, we use the representative normalization
\[
 U=B_{1/2},\qquad D_{\mathrm{cmp}}=B_{3/2},
 \qquad D_{\mathrm{out}}=B_2.
\]
Let \(v_0\) solve
\[
 G(D^2v_0)=c_*,\qquad v_0=v\quad\text{on }\partial B_{3/2},
\]
where \(0<c_*<\lambda/2\).
Existence follows from \cref{prop:dirichlet}, because \(v\) itself is a strict admissible subsolution for the constant equation; comparison gives \(v\le v_0\).
Since both are subharmonic and have the same boundary trace,
\[
 \norm{v_0}_{L^\infty(B_{3/2})}\le A.
\]
By the interior gradient estimate for the equation with constant data in \cite[Section~3]{ChouWang2001},
\begin{equation}
 \norm{Dv_0}_{L^\infty(B_1)}\le C(n,c_*,A).
\label{eq:constant-gradient}
\end{equation}

We first obtain a positive gap without differentiating the variable right side.
Put \(q=v_0-v\ge0\).
Quadratic polarization gives
\begin{equation}
 \Slin\!\left(D^2\frac{v+v_0}{2}\right)^{ij}q_{ij}
 =c_*^2-G(D^2v)^2\le-c_1<0.
\label{eq:gap-equation}
\end{equation}
The identity \(\partial_i\Slin(D^2((v+v_0)/2))^{ij}=0\) holds.
Choose a nonnegative bump \(\eta_r\in C_c^\infty(B_r)\) with
\[
 \int\eta_r=r^{n+2},\qquad \abs{D^2\eta_r}\le C.
\]
Twice integrating \eqref{eq:gap-equation} by parts gives
\begin{equation}
 c_1r^{n+2}
 \le C\sup_{B_r}q\int_{B_r}
 \tr\Slin\!\left(D^2\frac{v+v_0}{2}\right).
\label{eq:gap-mean-value}
\end{equation}
For a fixed cutoff \(\chi=1\) on \(B_r\),
\[
 \int_{B_r}\tr\Slin\!\left(D^2\frac{v+v_0}{2}\right)
 \le C\int\chi(\Delta v+\Delta v_0)
 =C\int(v+v_0)\Delta\chi\le C(A).
\]
Thus \(q\) is bounded below by a fixed positive constant at some point.

We next enlarge the pointwise gap to a ball.
\Cref{prop:modulus} gives a \(C^\beta\) modulus for \(v\) independent of derivatives of the right side.
Equation \eqref{eq:constant-gradient} gives a Lipschitz modulus for \(v_0\).
Hence \(q\) has a fixed \(C^\beta\) modulus, and the pointwise lower bound from \eqref{eq:gap-mean-value} holds on a ball:
\begin{equation}
 q\ge\delta_0>0\quad\text{on }B_{r_0}(x_*),
\label{eq:local-gap}
\end{equation}
with fixed \(r_0,\delta_0\).

Finally, we propagate the gap and keep the comparison component away from the outer boundary.
After normalizing the constant data, concavity gives
\[
 \langle DG(D^2v_0),D^2q\rangle
 \le G(D^2v_0)-G(D^2v)\le c_*-\lambda<0.
\]
Multiplying by the fixed factor \(2c_*\) gives \(L_{v_0}q\le-c_2<0\) for the linearization used in \cref{prop:gap}.
\Cref{prop:gap}, applied along finitely many fixed overlapping ball chains, propagates \eqref{eq:local-gap} to
\begin{equation}
 v_0-v\ge c_0>0\quad\text{on }B_{1/2}.
\label{eq:uniform-gap}
\end{equation}
To keep the eventual component away from the outer boundary, let \(h\) be the harmonic function in \(B_{3/2}\) with boundary trace \(v\).
Since \(\Delta v,\Delta v_0>0\),
\[
 0\le v_0-v\le h-v.
\]
The Poisson formula and the \(C^\beta\) modulus give
\begin{equation}
 h(x)-v(x)\le C\,\dist(x,\partial B_{3/2})^\beta.
\label{eq:boundary-decay}
\end{equation}

Choose a regular value \(\delta\in(c_0/4,c_0/2)\) by Sard's theorem, and let \(\Omega\) be the component of \(\{v_0-v>\delta\}\) containing \(B_{1/2}\).
Estimate \eqref{eq:boundary-decay} keeps this component a fixed distance from the outer boundary.
Set
\[
 W=v_0-\delta.
\]
The same interior gradient estimate, now applied on a fixed smaller ball, bounds \(Dv_0\) throughout \(\Omega\).
Then \eqref{eq:component-inclusions}--\eqref{eq:component-bounds} follow from \eqref{eq:constant-gradient}, \eqref{eq:uniform-gap}, and the construction.
\end{proof}
We need two quantitative consequences of the calculations in Chou--Wang \cite[Sections~3--4]{ChouWang2001}.
Their theorem statements do not record the precise dependence needed here.
We keep the parameters in their calculations (3.7)--(3.13) and (4.4)--(4.13).

\begin{lemma}\label{lem:gradient-bound}
Let \(V_0\Subset V\Subset\{\zeta=1\}\) be fixed domains, and let \(v\) be a smooth admissible solution of
\[
 G(D^2v)+(g-b)\eta(\Delta v/P)=g
 \quad\hbox{in }V.
\]
Assume
\[
 0<c_0\le b\le g\le c_1,\qquad
 \norm v_{L^\infty(V)}\le A_0.
\]
Put \(L=\norm{D(b^2)}_{L^\infty(V)}\).
Then
\begin{equation}
 \quad
 \norm{Dv}_{L^\infty(V_0)}
 \le C\left(1+P^{1/2}+L^{1/5}\right),\quad
\label{eq:gradient-estimate}
\end{equation}
where \(C\) depends only on \(n,c_0,c_1,A_0\), the fixed domain gap, and \(\eta\).
\end{lemma}

\begin{proof}
Fix a ball \(B_{2r}\Subset V\) such that \(B_r\) meets \(V_0\).
Choose \(\chi\in C_c^\infty(B_{2r})\) such that
\[
0\le \chi\le1,\qquad
\chi\equiv1\quad\hbox{on }B_r,
\]
and, on \(\{\chi>0\}\),
\[
\frac{|D\chi|^2}{\chi}+|D^2\chi|\le C.
\]
For instance, one may take the square of a standard cutoff.  Consequently,
if
\[
q_i=(\log\chi)_i=\frac{\chi_i}{\chi},
\]
then
\begin{equation}
 |q|^2+|D^2\log\chi|
 \le \frac{C}{\chi}.
 \label{eq:grad-cutoff}
\end{equation}

Set
\[
M_0=4A_0+1,
\qquad
\Phi(v)=(M_0-v)^{-1/2}.
\]
Since \(|v|\le A_0\),
\[
3A_0+1\le M_0-v\le5A_0+1.
\]
Writing
\[
a:=\frac{\Phi'}{\Phi}
   =\frac{1}{2(M_0-v)},
\]
we therefore have, with comparison constants depending only on \(A_0\),
\begin{equation}
 a\sim1,
 \qquad
 \beta:=
 \frac{\Phi''}{\Phi}
 -2\left(\frac{\Phi'}{\Phi}\right)^2
 =\frac{1}{4(M_0-v)^2}
 \sim1.
 \label{eq:grad-weight}
\end{equation}

Consider
\[
\mathcal G(x,\xi)
 =\chi(x)\Phi(v(x))v_\xi(x),
 \qquad x\in B_{2r},\quad \xi\in\mathbb S^{n-1},
\]
where the sign of \(\xi\) is chosen so that \(v_\xi\ge0\).
Equivalently,
\[
\max_{\xi\in\mathbb S^{n-1}}\mathcal G(x,\xi)
 =\chi(x)\Phi(v(x))|Dv(x)|.
\]
Let \((x_0,\xi_0)\) be a positive maximum.  Since \(\chi\) vanishes
on the boundary of its support, \(x_0\) is an interior point.
After a rotation we may assume
\[
\xi_0=e_1,
\qquad
Dv(x_0)=v_1e_1,
\qquad
v_1=|Dv(x_0)|>0.
\]
In what follows all quantities are evaluated at \(x_0\).

Differentiating $ \log\mathcal G $ once gives
\begin{equation}
 \frac{v_{1i}}{v_1}+a v_i+q_i=0.
 \label{eq:gradient-first-derivative}
\end{equation}
In particular, since \(v_i=0\) for \(i\ge2\),
\[
v_{11}
 =-a v_1^2-q_1v_1.
\]
By \eqref{eq:grad-cutoff} and Young's inequality,
\[
|q_1|v_1
 \le \frac a2v_1^2+\frac{C}{a\chi},
\]
and hence
\begin{equation}
 v_{11}
 \le -\frac a2v_1^2+\frac{C}{\chi}.
 \label{eq:grad-v11-pre}
\end{equation}
Thus, after increasing a constant \(C_*=C_*(A_0,r)\), there are
two alternatives:
\begin{equation}
 \chi v_1^2\le C_*,
 \label{eq:grad-cutoff-alt}
\end{equation}
or
\begin{equation}
 v_{11}\le-c\,v_1^2<0,
 \label{eq:negative-directional-hessian}
\end{equation}
where here and below \(c,C>0\) may depend on
\(n,A_0,r\), but not on \(P\) or \(b\).

We first consider the nonplateau case
\[
T:=\Delta v<\tau_*P.
\]
If \eqref{eq:grad-cutoff-alt} holds there is nothing further to prove.
Otherwise \eqref{eq:negative-directional-hessian} holds.  Since
\[
\lambda_{\min}(D^2v)\le v_{11}
\]
and, by \cref{lem:newton},
\[
\lambda_{\min}(D^2v)>-(n-2)T,
\]
we obtain
\[
c v_1^2
 \le -v_{11}
 <(n-2)T
 <(n-2)\tau_*P.
\]
Therefore
\begin{equation}
 \chi v_1^2\le C(1+P)
 \qquad\text{if }T<\tau_*P.
 \label{eq:grad-low-trace}
\end{equation}

It remains to consider
\[
T\ge\tau_*P.
\]
Since \(\eta\equiv1\) on \([\tau_*,\infty)\), and all its positive
derivatives vanish at \(t=\tau_*\), the equation and its first spatial
derivatives at \(x_0\) agree with those of
\begin{equation}
 \sigma_2(D^2v)=b^2.
 \label{eq:plateau-equation}
\end{equation}
Put
\[
h=b^2,
\qquad
F^{ij}:=\frac{\partial\sigma_2}{\partial v_{ij}}
       =T\delta_{ij}-v_{ij}.
\]
Because \(D^2v\in\Gamma_2\), the matrix \(F=(F^{ij})\) is positive
definite.  Moreover,
\begin{equation}
 \Sigma_F:=\sum_iF^{ii}
 =(n-1)T,
 \label{eq:grad-linearized-trace}
\end{equation}
and the homogeneity of \(\sigma_2\) gives
\begin{equation}
 F^{ij}v_{ij}=2\sigma_2(D^2v)=2h.
 \label{eq:grad-euler}
\end{equation}
Differentiating \eqref{eq:plateau-equation} in the \(e_1\)-direction yields
\begin{equation}
 F^{ij}v_{ij1}=h_1.
 \label{eq:grad-diff-eqn}
\end{equation}

We now carry out the second derivative calculation explicitly.
At the maximum point,
\[
0\ge
F^{ij}(\log\mathcal G)_{ij}.
\]
Since
\[
(\log\Phi(v))_{ij}
 =a v_{ij}
  +\left(
      \frac{\Phi''}{\Phi}-a^2
    \right)v_iv_j,
\]
we have, using
\eqref{eq:grad-euler}--\eqref{eq:grad-diff-eqn},
\begin{align}
0
&\ge
 \frac{h_1}{v_1}
 -F^{ij}\frac{v_{1i}v_{1j}}{v_1^2}
 +2ah
 +\left(
     \frac{\Phi''}{\Phi}-a^2
   \right)F^{ij}v_iv_j
 +F^{ij}(\log\chi)_{ij}.
\label{eq:grad-second-raw}
\end{align}
By the first derivative identity \eqref{eq:gradient-first-derivative},
\[
\frac{v_{1i}}{v_1}=-(av_i+q_i),
\]
and hence
\begin{align*}
&-F^{ij}\frac{v_{1i}v_{1j}}{v_1^2}
 +\left(
     \frac{\Phi''}{\Phi}-a^2
   \right)F^{ij}v_iv_j
\\
&\qquad
=
\left(
 \frac{\Phi''}{\Phi}-2a^2
\right)F^{ij}v_iv_j
 -2aF^{ij}v_iq_j
 -F^{ij}q_iq_j
\\
&\qquad
=
\beta F^{ij}v_iv_j
 -2aF^{ij}v_iq_j
 -F^{ij}q_iq_j.
\end{align*}
Since \(F>0\), Cauchy's inequality with respect to the quadratic
form \(F\) gives
\[
2a\bigl|F^{ij}v_iq_j\bigr|
\le
\frac{\beta}{2}F^{ij}v_iv_j
 +\frac{2a^2}{\beta}F^{ij}q_iq_j.
\]
By \eqref{eq:grad-weight}, \(a^2/\beta\) is uniformly bounded.
Therefore
\begin{equation}
-F^{ij}\frac{v_{1i}v_{1j}}{v_1^2}
 +\left(
     \frac{\Phi''}{\Phi}-a^2
   \right)F^{ij}v_iv_j
\ge
\frac{\beta}{2}F^{ij}v_iv_j
 -C F^{ij}q_iq_j.
\label{eq:grad-square-complete}
\end{equation}
Because \(Dv=v_1e_1\) at \(x_0\),
\[
F^{ij}v_iv_j=F^{11}v_1^2.
\]
Also, positivity of \(F\), \eqref{eq:grad-cutoff}, and
\eqref{eq:grad-linearized-trace} imply
\[
F^{ij}q_iq_j
 \le\Sigma_F|q|^2
 \le\frac{C\Sigma_F}{\chi},
\]
and
\[
F^{ij}(\log\chi)_{ij}
 \ge-\Sigma_F|D^2\log\chi|
 \ge-\frac{C\Sigma_F}{\chi}.
\]
Substituting these estimates into
\eqref{eq:grad-second-raw}, and using \(a>0\), \(h>0\), gives
\begin{equation}
 \frac{\beta}{2}F^{11}v_1^2
 \le
 \frac{|Dh|}{v_1}
 +\frac{C\Sigma_F}{\chi}.
 \label{eq:grad-master}
\end{equation}

We may again assume that the cutoff alternative
\eqref{eq:grad-cutoff-alt} does not hold, with \(C_*\) enlarged
once more if necessary.  Then \eqref{eq:negative-directional-hessian} holds, and therefore
\[
F^{11}=T-v_{11}\ge T
       =\frac{\Sigma_F}{n-1}.
\]
Thus \eqref{eq:grad-master}, after multiplication by \(v_1\), yields
\begin{equation}
 c\Sigma_F v_1^3
 \le |Dh|+\frac{C\Sigma_F v_1}{\chi}.
 \label{eq:grad-before-absorb}
\end{equation}
If \(C_*\) in \eqref{eq:grad-cutoff-alt} is chosen sufficiently large,
then its failure implies
\[
\frac{C\Sigma_F v_1}{\chi}
 \le \frac c2\Sigma_F v_1^3.
\]
Consequently
\begin{equation}
 \Sigma_F v_1^3\le C|Dh|.
 \label{eq:gradient-cubic-bound}
\end{equation}

It remains only to give the elementary lower bound for
\(\Sigma_F\).  From \eqref{eq:negative-directional-hessian},
\[
\lambda_{\min}(D^2v)
 \le v_{11}
 \le-cv_1^2.
\]
On the other hand, \cref{lem:newton} gives
\[
\lambda_{\min}(D^2v)>-(n-2)T.
\]
Hence
\[
T\ge c v_1^2,
\]
and therefore, by \eqref{eq:grad-linearized-trace},
\begin{equation}
 \Sigma_F=(n-1)T\ge c v_1^2.
 \label{eq:linearized-trace-lower}
\end{equation}
Combining \eqref{eq:gradient-cubic-bound} and \eqref{eq:linearized-trace-lower}, and recalling that
\[
|Dh|=|D(b^2)|\le L,
\]
we conclude that in the large-gradient plateau alternative
\begin{equation}
 v_1^5\le C L.
 \label{eq:gradient-fifth-power-bound}
\end{equation}

We have therefore proved that at the maximum point \(x_0\) one of
the following holds:
\[
\chi(x_0)v_1(x_0)^2\le C,
\qquad
\chi(x_0)v_1(x_0)^2\le C(1+P),
\qquad\text{or}\qquad
v_1(x_0)^5\le C(1+L).
\]
Since \(\Phi\) is bounded above and below by positive constants depending
only on \(A_0\), for every \(x\in B_r\),
\begin{align*}
|Dv(x)|
&\le
C\,\mathcal G(x_0,\xi_0)
\\
&\le
C\,\chi(x_0)v_1(x_0)
\\
&\le
C\left(1+P^{1/2}+L^{1/5}\right).
\end{align*}
Thus
\[
\|Dv\|_{L^\infty(B_r)}
\le
C\left(1+P^{1/2}+L^{1/5}\right).
\]
Finally, a finite covering of
\(\overline{V_0}\) by such balls \(B_r\), with their doubled balls
contained in \(V\), gives
\[
\|Dv\|_{L^\infty(V_0)}
\le
C\left(1+P^{1/2}+L^{1/5}\right).
\]
This proves the lemma.
\end{proof}

\begin{lemma}
\label{lem:hessian-bound-v}
Let \(\Omega\) be a smooth bounded domain and let \(v,W\) be smooth admissible functions on a neighborhood of \(\overline\Omega\).
Suppose
\[
 \rho=W-v>0\ \hbox{in }\Omega,\qquad
 \rho=0\ \hbox{on }\partial\Omega,
\]
and let \(U\Subset\Omega\) satisfy
\[
 \rho\ge\rho_0>0\quad\hbox{on }U,\qquad
 \rho\le\rho_1\quad\hbox{on }\Omega.
\]
Assume that the global maximum of the auxiliary function introduced below occurs at a point where, through second order, the equation is
\[
 G(D^2v)=b(x),\qquad 0<c_0\le b\le c_1.
\]
Set
\[
 K=1+\norm{Dv}_{L^\infty(\Omega)}
       +\norm{DW}_{L^\infty(\Omega)},\qquad
 L=\norm{D(b^2)}_\infty,\qquad
 J=\norm{D^2(b^2)}_\infty.
\]
Then
\begin{equation}
 \quad
 \sup_U\Delta v
 \le C\left(1+K^2+KL+K\sqrt J\right),\quad
\label{eq:hessian-estimate}
\end{equation}
where \(C\) depends only on \(n,c_0,c_1,\rho_0,\rho_1\) and fixed domain geometry.
\end{lemma}

\begin{proof}
\[
 \varphi(t)=(1-t/K^2)^{-1/8},
 \qquad 0\le t\le\norm{Dv}_\infty^2/2.
\]
Since \(K\ge1+\norm{Dv}_\infty\),
\[
 1\le\varphi\le C,\qquad
 \frac{\varphi'}{\varphi}\sim K^{-2},\qquad
 \frac{\varphi''}{\varphi}
 -3\left(\frac{\varphi'}{\varphi}\right)^2\ge cK^{-4}.
\]
Following \cite[Section~4]{ChouWang2001}, consider
\begin{equation}
 \Psi(x,\xi)=
 \rho(x)^4\varphi(\abs{Dv(x)}^2/2)v_{\xi\xi}(x),
 \qquad \xi\in\Sn.
\label{eq:hessian-test}
\end{equation}
At an interior positive maximum, use the standard eigenvalue perturbation argument, rotate coordinates so \(D^2v\) is diagonal, and take \(\xi=e_1\), where \(\lambda_1=v_{11}\) is the largest eigenvalue.
Write \(\Sigma_G=\sum_iG^{ii}\).
Differentiating \(\log\Psi\) once gives
\begin{equation}
 0=4\frac{\rho_i}{\rho}
   +\frac{\varphi'}{\varphi}v_kv_{ki}
   +\frac{v_{11i}}{\lambda_1}.
\label{eq:hessian-first-derivative}
\end{equation}
Differentiating twice and contracting with \(G^{ii}\) gives
\begin{equation}
\begin{aligned}
0\ge {}&
4G^{ii}\left(\frac{\rho_{ii}}{\rho}
                  -\frac{\rho_i^2}{\rho^2}\right)\\
&+\frac{\varphi'}{\varphi}G^{ii}
       \left(v_{ki}^2+v_kv_{kii}\right)
+\left(\frac{\varphi''}{\varphi}
       -\left(\frac{\varphi'}{\varphi}\right)^2\right)
 G^{ii}(v_kv_{ki})^2\\
&+\frac{G^{ii}v_{11ii}}{\lambda_1}
 -\frac{G^{ii}v_{11i}^2}{\lambda_1^2}.
\end{aligned}
\label{eq:hessian-second-derivative}
\end{equation}
The differentiated equation is
\[
 G^{ii}v_{iik}=b_k,\qquad
 G^{ii}v_{ii11}
 +G^{ij,rs}v_{ij1}v_{rs1}=b_{11}.
\]
Concavity makes the quadratic term in the second identity favorable.
Also, by concavity and homogeneity,
\begin{equation}
 G^{ii}v_{ii}=b,\qquad
 G^{ii}W_{ii}\ge G(D^2W),
\label{eq:rho-concavity}
\end{equation}
so \(D^2W\) enters the contracted \(G^{ii}\rho_{ii}\) term with a favorable sign; no upper bound for \(D^2W\) is needed.

Since \(b\ge c_0\), the definitions of \(L\) and \(J\) give \(\norm{Db}_\infty\le CL\) and \(\norm{D^2b}_\infty\le C(J+L^2)\).
The terms in \eqref{eq:hessian-second-derivative} containing derivatives of the right side satisfy
\[
 \left|\frac{\varphi'}{\varphi}v_kb_k\right|
 \le C\frac{L}{K},
 \qquad
 \frac{b_{11}}{\lambda_1}\ge-\frac{C(J+L^2)}{\lambda_1}.
\]
Moreover, \(\abs{D\rho}\le K\), while homogeneity and concavity give
\[
 G^{ii}\rho_{ii}
 =G^{ii}(W_{ii}-v_{ii})
 \ge G(D^2W)-b\ge-C.
\]
Thus the terms containing \(D\rho\) are bounded by \(C\Sigma_G K^2\rho^{-2}\) in the nonseparated case and by \(CG^{11}K^2\rho^{-2}\) in the separated case.
These estimates explain all nonfixed quantities in \eqref{eq:hessian-case-one} and \eqref{eq:hessian-case-two}; no derivative of the original function \(g\) occurs.

We now retain the parameters in the completion of squares.
This is the case \(k=2\) of \cite[(4.4)--(4.13)]{ChouWang2001}.
If the second eigenvalue is a fixed fraction of \(\lambda_1\), insertion of \eqref{eq:hessian-first-derivative}--\eqref{eq:rho-concavity} into \eqref{eq:hessian-second-derivative}, followed by the completion of squares in their equations (4.7)--(4.10), gives
\begin{equation}
0\ge  cK^{-2}\Sigma_G\lambda_1^2
 -C\Sigma_G K^2\rho^{-2} -C(1+\rho^{-1})-CL/K-C(J+L^2)/\lambda_1.
\label{eq:hessian-case-one}
\end{equation}
Here the first term comes from \(\varphi'/\varphi\sim K^{-2}\).
The second comes from the terms containing \(D\rho\) after \eqref{eq:hessian-first-derivative} is used.
If that second term cannot be absorbed by the first, then
\begin{equation}
 \rho^2\lambda_1^2\le CK^4,
\label{eq:direct-hessian-bound}
\end{equation}
already.
If it is absorbed, then
\[
 \Sigma_G=\frac{(n-1)\Delta v}{2b}\ge c\lambda_1.
\]
For \(\lambda_1\ge1\), \eqref{eq:hessian-case-one} therefore gives an inequality stronger than
\begin{equation}
 \frac{\lambda_1}{K^2}
 \le C\left(\rho^{-1}+\frac{L}{K}
                  +\frac{J+L^2}{\lambda_1}+1\right),
\label{eq:hessian-absorption}
\end{equation}
while \(\lambda_1<1\) is harmless.

When the second eigenvalue is small compared with \(\lambda_1\), equations (4.12)--(4.13) of the cited paper give, before absorption,
\begin{equation}
\begin{aligned}
0\ge {}&
 cK^{-2}G^{11}\lambda_1^2
 -CG^{11}K^2\rho^{-2}\\
&-C\left(1+\rho^{-1}+L/K+(J+L^2)/\lambda_1\right).
\end{aligned}
\label{eq:hessian-case-two}
\end{equation}
Again, if the term containing \(D\rho\) cannot be absorbed, we obtain \eqref{eq:direct-hessian-bound}.
Otherwise the identity
\[
 \lambda_1(\Delta v-\lambda_1)\ge c\sigma_2(D^2v)=c b^2
\]
implies
\[
 G^{11}\lambda_1^2
 =\frac{(\Delta v-\lambda_1)\lambda_1^2}{2b}\ge cb\lambda_1,
\]
and \eqref{eq:hessian-case-two} gives exactly \eqref{eq:hessian-absorption}.
Thus both eigenvalue cases give the stated dichotomy, with no unspecified dependence on \(K\).

To finish the scalar algebra, multiply \eqref{eq:hessian-absorption} by \(\lambda_1\) and solve the resulting quadratic inequality:
\[
 \lambda_1
 \le C\left(K^2(1+\rho^{-1})+KL+K\sqrt J\right).
\]
After multiplication by \(\rho^4\varphi\), the powers \(\rho^4,\rho^3\) are bounded by constants depending only on \(\rho_1\).
In the alternative \eqref{eq:direct-hessian-bound}, \(\rho^4\lambda_1\le CK^2\rho^3\).
Hence
\begin{equation}
 \max_{\overline\Omega\times\Sn}\Psi
 \le C\left(1+K^2+KL+K\sqrt J\right).
\label{eq:hessian-test-bound}
\end{equation}
Finally, on \(U\), \(\rho\ge\rho_0\), \(\varphi\ge1\), and \(\Delta v\le n\lambda_{\max}\).
Thus \eqref{eq:hessian-test-bound} and the preceding estimate imply \eqref{eq:hessian-estimate}.
\end{proof}

\begin{remark}
Theorem~4.1 of \cite{ChouWang2001} is stated in a globally strictly \((k-1)\)-convex domain.
The present use is the local interior maximum calculation in its proof.
Strict boundary convexity does not occur in equations (4.4)--(4.13); the factor \(\rho^4\) forces the maximum away from \(\partial\Omega\).
Thus the proof of \cref{lem:hessian-bound-v} uses only the interior calculation in that paper.
\end{remark}

\begin{lemma}
\label{lem:maximum-point}
In the setting of \cref{lem:component}, suppose \(v\) satisfies \eqref{eq:adaptive-dirichlet} on a fixed neighborhood of \(\overline\Omega\), and suppose
\begin{equation}
 \Delta v\ge cM\quad\hbox{at some point of }U.
\label{eq:large-trace-point}
\end{equation}
If \(P=o(M)\), then, for large \(M\), the global maximum of \(\Psi\) in \eqref{eq:hessian-test} is an interior point of \(\{\Delta v\ge\tau_*P\}\).
At that point \(v\) satisfies \(G(D^2v)=b\) through second order.
\end{lemma}

\begin{proof}
On \(U\), \(\rho\ge\rho_0\), and \(\lambda_{\max}\ge\Delta v/n\); hence \eqref{eq:large-trace-point} gives \(\max\Psi\ge cM\).
At every point where \(\Delta v<\tau_*P\),
\[
 \lambda_{\max}<\Delta v<\tau_*P,
\]
so \(\Psi\le CP=o(M)\).
Therefore the maximum cannot occur there.
The boundary is excluded because \(\rho=0\) there.
Since \(\overline\Omega\Subset\{\zeta=1\}\) and every derivative of \(\eta\) vanishes at and above \(\tau_*\), the first two derivatives of the equation at the maximum agree with those of \(G(D^2v)=b\).
\end{proof}

\begin{proof}[Proof of \cref{prop:comparison-upper-trace}]
Apply \cref{lem:component} to \(v\) using \(U\Subset D_{\rm cmp}\Subset D_{\rm flat}\).
It gives \(W,\Omega\), and \(\rho=W-v\), with the uniform bounds in \eqref{eq:component-inclusions}--\eqref{eq:component-bounds}.
The possible components are covered by finitely many balls of one fixed radius contained in \(D_{\rm flat}\).
Hence \cref{lem:gradient-bound} gives
\[
 K\le C\left(1+P^{1/2}+L^{1/5}\right).
\]
By \cref{lem:maximum-point}, the maximum of \eqref{eq:hessian-test} occurs where \(G(D^2v)=b\) through second order.
We may therefore apply \cref{lem:hessian-bound-v}, which gives the asserted upper bound for \(\Delta v\) on \(U\).
\end{proof}

\section{Completion of the proofs}
\label{sec:contradiction}

The preceding sections give the lower and upper trace estimates needed for the contradiction.
We now combine them to prove the Hessian bound, derive the \(C^{2,\alpha}\) estimate, and pass to viscosity solutions.

\subsection{The Hessian bound}
\begin{theorem}
\label{thm:hessian-bound}
Under the hypotheses of \cref{thm:smooth},
\begin{equation}
 \norm{D^2u}_{L^\infty(B_{3/4})}
 \le C\left(n,\alpha,f_0,\norm f_{C^\alpha(B_2)},
                  \norm u_{L^\infty(B_2)}\right).
\label{eq:interior-hessian-bound}
\end{equation}
\end{theorem}

\begin{proof}
Assume the assertion is false.
Then there exist smooth admissible solutions \(u_j\), with
\begin{equation}
 \norm{u_j}_{L^\infty(B_2)}\le A,\qquad
 f_j\ge f_0,\qquad
 \norm{f_j}_{C^\alpha(B_2)}\le\Lambda,
\label{eq:contradiction-data}
\end{equation}
but
\[
 \sup_{B_{3/4}}\abs{D^2u_j}\longrightarrow\infty.
\]
Since \(\abs{D^2u_j}<\Delta u_j\), the traces $\Delta u_j$ are unbounded.
Cover \(\overline{B_{3/4}}\) by finitely many balls whose fourfold dilates are compactly contained in \(B_2\).
After passing to a subsequence, the trace is unbounded on the concentric ball with half the radius of one fixed member of this cover.
A fixed translation and dilation reduce that ball to the setting of \cref{lem:stopping}; the bounds for the right side and the \(L^\infty\) norm of the solution change only by fixed factors.
We retain the original notation after this reduction.
Set \(g_j=\sqrt{f_j}\).
The functions \(g_j\) have common positive upper and lower bounds and a common \(C^\alpha\) bound.

By \cref{sec:large-trace,sec:large-trace-set}, after passing to a subsequence we have
\[
 M_j\to\infty,\qquad
 E_j\Subset U,\qquad
 \Delta u_j\ge cM_j\ \hbox{on }E_j,\qquad
 \abs{E_j}\ge c\frac{S_j}{M_j},
\]
where
\begin{equation}
 S_j=\frac{\ell_j^{\,n-1}}{F(M_j)^n},\qquad
 S_j^{-1}\le F(M_j)^n,\qquad
 F(M)=(\log(e+M))^2.
\label{eq:large-set-scale}
\end{equation}
The fixed domains and cutoff are those of \eqref{eq:nested-domains}--\eqref{eq:fixed-cutoff}.

Choose
\begin{equation}
 \eps_j=M_j^{-1/8},\qquad
 P_j=M_j^{\,1-\alpha/8}.\qquad
\label{eq:parameter-choice}
\end{equation}
Let \(b_j=b_{\eps_j}\) be supplied by \cref{lem:lower-mollification}, and write
\[
 \delta_j=\norm{g_j-b_j}_\infty,\qquad
 L_j=\norm{D(b_j^2)}_\infty,\qquad
 J_j=\norm{D^2(b_j^2)}_\infty.
\]
Uniformly in \(j\),
\begin{equation}
 \delta_j\le CM_j^{-\alpha/8},\qquad
 L_j\le CM_j^{(1-\alpha)/8},\qquad
 J_j\le CM_j^{(2-\alpha)/8}.
\label{eq:smoothing-rates}
\end{equation}

By \cref{thm:dirichlet-problem}, there is a smooth admissible \(v_j\) solving
\[
 G(D^2v_j)+\zeta(g_j-b_j)
       \eta(\Delta v_j/P_j)=g_j,\qquad
 v_j=u_j\quad\hbox{on }\partial D.
\]
By \eqref{eq:solution-comparison}, the functions \(v_j\) are uniformly bounded in \(L^\infty\) in terms of \eqref{eq:contradiction-data}.
Moreover,
\begin{equation}
 \frac{M_j\delta_j^2}{P_jS_j}
 \le C M_j^{-\alpha/8}(\log(e+M_j))^{2n}
 \longrightarrow0.
\label{eq:transfer-smallness}
\end{equation}
Thus \cref{prop:comparison-lower-trace} gives a subset \(E'_j\subset E_j\), of measure at least \(\abs{E_j}/2\), on which
\begin{equation}
 \Delta v_j\ge cM_j.
\label{eq:transferred-trace}
\end{equation}

Since \(P_j/M_j=M_j^{-\alpha/8}\to0\), we may apply \cref{prop:comparison-upper-trace}.
It gives \(W_j,\Omega_j\) and
\begin{equation}
 K_j:=1+\norm{Dv_j}_{L^\infty(\Omega_j)}
        +\norm{DW_j}_{L^\infty(\Omega_j)}
 \le C\left(1+P_j^{1/2}+L_j^{1/5}\right).
\label{eq:comparison-gradient-bound}
\end{equation}
It is essential here that the gradient is controlled on the entire comparison component, not merely on \(E'_j\).
The same proposition gives
\begin{equation}
 cM_j\le\sup_U\Delta v_j
 \le C\left(1+K_j^2+K_jL_j+K_j\sqrt{J_j}\right).
\label{eq:contradiction-inequality}
\end{equation}

It remains to check the powers.
From \eqref{eq:parameter-choice}, \eqref{eq:smoothing-rates}, and \eqref{eq:comparison-gradient-bound},
\[
 K_j\le C M_j^{\,1/2-\alpha/16}.
\]
Consequently
\[
 K_j^2\le CM_j^{1-\alpha/8},\qquad
 K_jL_j\le CM_j^{5/8-3\alpha/16},\qquad
 K_j\sqrt{J_j}\le CM_j^{5/8-\alpha/8}.
\]
All three powers are strictly smaller than one for every fixed \(0<\alpha<1\).
Hence the right side of \eqref{eq:contradiction-inequality} is \(o(M_j)\), contradicting its left side.
This proves \eqref{eq:interior-hessian-bound}.
\end{proof}

\begin{remark}
The power gap is not uniform as \(\alpha\downarrow0\), and no such uniformity is claimed.
In fact, we first fix \(\alpha\in(0,1)\), then send \(M_j\to\infty\).
The logarithmic factor in \eqref{eq:transfer-smallness} is dominated by the strictly negative power \(-\alpha/8\) for every fixed \(\alpha\).
\end{remark}

\subsection{The \texorpdfstring{\(C^{2,\alpha}\)}{C2-alpha} estimate}

\begin{proposition}
\label{prop:perturbative-schauder}
Fix \(0<\alpha<1\) and ellipticity constants \(0<\lambda\le\Lambda\).
There is \(\delta>0\) with the following property.
Let \(F\in C^1(\operatorname{Sym}(n))\) be uniformly elliptic, \(F(0)=0\), and suppose that for some constant matrix \(a=(a^{ij})\), with \(\lambda I\le a\le\Lambda I\),
\begin{equation}
 \sup_N\norm{DF(N)-a}\le\delta.
\label{eq:operator-closeness}
\end{equation}
If
\[
 F(D^2w)=q\quad\hbox{in }B_1,\qquad
 \norm w_{L^\infty(B_1)}\le1,\qquad
 \norm q_{C^\alpha(B_1)}\le\delta,
\]
then
\[
 \norm w_{C^{2,\alpha}(B_{1/2})}\le C(n,\lambda,\Lambda,\alpha).
\]
The same conclusion holds if \eqref{eq:operator-closeness} is known only on the matrix ball visited by the equation and the operator is extended outside that ball without changing the ellipticity constants or the \(C^1\) closeness.
\end{proposition}

This is the perturbation result proved in \cite{Caffarelli1989}.

\begin{proof}[Proof of \cref{thm:smooth}]
\Cref{thm:hessian-bound} first gives a Hessian bound depending only on the data on \(B_{3/4}\).  If \(\lambda_i\) is an eigenvalue of \(D^2u\), then \(\tr(D^2u)-\lambda_i\) is an eigenvalue of \(\Slin(D^2u)\).  The estimate
\[
 \sigma_2(D^2u)\le
 \tr(D^2u)\big(\tr(D^2u)-\lambda_i\big),
 \qquad
 \tr(D^2u)-\lambda_i<2\tr(D^2u),
\]
together with \(f_0\le\sigma_2(D^2u)\le\norm f_\infty\), shows that \(DG(D^2u)=\Slin(D^2u)/(2\sqrt f)\) has fixed positive upper and lower bounds.
Let
\[
 K_0=\overline{\{D^2u(x):x\in B_{3/4}\}}\Subset\Gtwo.
\]
To apply Evans--Krylov without assuming that \(G\) is defined outside its cone, extend it to all symmetric matrices by its supporting planes:
\[
 \widetilde G(N)
 =\inf_{A\in K_0}\{G(A)+\langle DG(A),N-A\rangle\}.
\]
The infimum of affine functions is concave.
Every slope \(DG(A)\) has the same ellipticity bounds, so \(\widetilde G\) is uniformly elliptic.
Concavity of \(G\) shows that every supporting plane lies above \(G\) on \(K_0\), while choosing \(A=N\) gives equality.
Thus
\[
 \widetilde G(D^2u)=G(D^2u)=\sqrt f
\quad\hbox{in }B_{3/4}.
\]
\Cref{prop:concave-estimates} gives, with a bound depending only on the data,
\begin{equation}
 D^2u\in C^{\beta_0}(B_{2/3})
\label{eq:seed-holder}
\end{equation}
for some \(\beta_0\in(0,\min\{\alpha,\bar\beta(n,\lambda,\Lambda)\})\).
Thus \(\beta_0\) depends on dimension, ellipticity, and \(\alpha\), as it must when the H\"older exponent of the right side is smaller than the universal Evans--Krylov exponent.

We now return to the original smooth operator \(G\), which is \(C^\infty\) on a fixed neighborhood of \(K_0\).
By \eqref{eq:seed-holder}, there is a radius \(r_0>0\), depending only on the data, such that for every \(x_0\in B_{1/2}\),
\[
 \osc_{B_{r_0}(x_0)}D^2u\le\delta_0,
\]
where \(\delta_0\) is chosen below.
Subtract the quadratic Taylor polynomial at \(x_0\) and rescale \(B_{r_0}(x_0)\) to \(B_1\) as follows.
Put \(A_0=D^2u(x_0)\) and define
\[
 w(y)=\frac{u(x_0+r_0y)-u(x_0)-r_0Du(x_0)\cdot y
 -\frac12r_0^2 y^TA_0y}{r_0^2},\qquad y\in B_1.
\]
Then \(w(0)=0\), \(Dw(0)=0\), and
\[
 D^2w(y)=D^2u(x_0+r_0y)-A_0,
 \qquad \norm w_{L^\infty(B_1)}\le C\delta_0.
\]
The last inequality follows by integrating \(D^2w\) twice along the segment from \(0\) to \(y\).
The equation becomes
\[
 F_{x_0}(D^2w)=q_{x_0}(y),
 \quad
 F_{x_0}(N):=G(A_0+N)-G(A_0),
 \quad
 q_{x_0}(y):=\sqrt{f(x_0+r_0y)}-\sqrt{f(x_0)}.
\]
On the matrix ball actually reached by \(D^2w\), Taylor's theorem gives
\[
 \norm{DG(A_0+N)-DG(A_0)}
 \le \norm{D^2G}_{L^\infty}\abs N\le C\delta_0,
\]
and \(a:=DG(A_0)\) has the fixed ellipticity constants.

Here is the promised global extension, written explicitly.
Because \(K_0\Subset\Gtwo\), after making \(\delta_0\) smaller the expression
\[
 H(N):=G(A_0+N)-G(A_0)-\langle a,N\rangle
\]
is defined for \(\abs N\le3\delta_0\).
Choose a smooth radial cutoff \(\chi\) which is one on \(\abs N\le\delta_0\), zero on \(\abs N\ge2\delta_0\), and satisfies \(\abs{D\chi}\le C/\delta_0\).
Set
\[
 \widehat{F}_{x_0}(N)=\langle a,N\rangle+\chi(N)H(N),
\]
where the product \(\chi H\) is declared to be zero outside the support of \(\chi\).
Since \(H(N)=O(\abs N^2)\) and \(DH(N)=O(\abs N)\),
\[
 \sup_{N\in\operatorname{Sym}(n)}
 \norm{D\widehat{F}_{x_0}(N)-a}\le C\delta_0.
\]
For \(\delta_0\) small this operator is globally uniformly elliptic, with slightly enlarged fixed constants, and it agrees with \(F_{x_0}\) on every matrix visited by \(D^2w\).

Moreover,
\[
 \norm{q_{x_0}}_{C^\alpha(B_1)}
 \le C r_0^\alpha[\sqrt f]_{C^\alpha(B_{3/4})}.
\]
Taking first \(\delta_0\) and then \(r_0\) small makes \(\norm w_{L^\infty(B_1)}\le1\) and puts all three quantities in the normalization of \cref{prop:perturbative-schauder}.
Scaling its estimate back and covering \(B_{1/2}\) by finitely many such balls yields
\begin{equation}
 [D^2u]_{C^\alpha(B_{1/2})}\le C.
\label{eq:exact-holder}
\end{equation}
Finally, an elementary interior finite difference estimate bounds \(\norm{Du}_{L^\infty(B_{1/2})}\) by \(\norm u_{L^\infty(B_{3/4})}+\norm{D^2u}_{L^\infty(B_{3/4})}\).
Together with \eqref{eq:interior-hessian-bound} and \eqref{eq:exact-holder}, this proves \eqref{eq:smooth-estimate}.
\end{proof}

\subsection{Viscosity solutions}
\label{sec:viscosity}
\begin{proposition}
\label{prop:guan-ball}
Let \(D\) be a Euclidean ball, let \(q\in C^\infty(\overline D)\) be strictly positive, and let \(\phi\in C^\infty(\partial D)\).
Then
\[
 G(D^2w)=q\quad\hbox{in }D,\qquad w=\phi\quad\hbox{on }\partial D
\]
has a unique smooth \(\Gtwo\)-admissible solution.
\end{proposition}

\begin{proof}
Choose a smooth extension \(\Psi\) of \(\phi\) to \(\overline D\), and set
\[
 \underline w=\Psi+A(\abs{x-x_D}^2-R^2).
\]
For \(A\) large, \(D^2\underline w=D^2\Psi+2AI\) is positive definite and
\[
 G(D^2\underline w)>\sup_Dq.
\]
Thus \(\underline w\) is a strict admissible subsolution with the required boundary value.
The operator is symmetric, elliptic, concave, homogeneous of degree one, and vanishes on \(\partial\Gtwo\).
Since \(D\) is a ball and \(\inf_Dq>0\), the remaining hypotheses of \cite[Theorem~1.1]{Guan2014} hold.
That theorem gives existence, and comparison gives uniqueness.
\end{proof}

\begin{proof}[Proof of \cref{thm:viscosity}]
Fix balls
\[
 D=B_R(x_0)\Subset D^+\Subset B_2.
\]
Set \(g=\sqrt f\).
Choose \(\eps_m\downarrow0\) smaller than \(\dist(D,\partial D^+)\), and define on \(\overline D\)
\[
 g_m=\rho_{\eps_m}*g,\qquad f_m=g_m^2.
\]
Only values of \(g\) in \(D^+\) enter the convolution.
Since the mollifier is nonnegative,
\begin{equation}
 g_m\ge\sqrt{f_0},\qquad
 g_m\to g\ \hbox{uniformly},\qquad
 \sup_m\norm{f_m}_{C^\alpha(D)}
 \le C(f_0,\norm f_{C^\alpha(D^+)}).
\label{eq:approximation-data}
\end{equation}
Choose \(\phi_m\in C^\infty(\partial D)\) converging uniformly to the boundary trace of \(u\).
By \cref{prop:guan-ball}, there is a smooth admissible solution
\begin{equation}
 G(D^2u_m)=g_m\quad\hbox{in }D,\qquad
 u_m=\phi_m\quad\hbox{on }\partial D.
\label{eq:smooth-approximants}
\end{equation}
The norms of the strict subsolutions and the high derivatives of \(\phi_m,g_m\) may diverge; none of them will be used.

We next prove uniform convergence and, at the same time, the $L^\infty$ bound required by the smooth theorem.
Put
\[
 \eta_m=\frac{\norm{g_m-g}_{L^\infty(D)}}{\sqrt{f_0}},
 \qquad a_m=1+\eta_m,\qquad b_m=1-\eta_m,
\]
discarding finitely many terms so that \(b_m>0\).
Let
\[
 S=\sup_Du,\qquad
 e_m=\norm{\phi_m-u}_{L^\infty(\partial D)}
\]
and define
\[
 U_m^-=a_mu-(a_m-1)S-e_m,\qquad
 U_m^+=b_mu+(1-b_m)S+e_m.
\]
Positive multiplication and addition of constants preserve \(\Phi_2(D)\), so \(U_m^-,U_m^+\in\Phi_2(D)\).
If \(\delta_m=\norm{g_m-g}_{L^\infty(D)}\), then
\[
 a_mg=g+\eta_mg\ge g+\delta_m\ge g_m,
 \qquad
 b_mg=g-\eta_mg\le g-\delta_m\le g_m.
\]
By \(\mu_2[u]=f\,dx\) and the homogeneity of the Hessian measure,
\[
 \mu_2[U_m^-]=a_m^2f\,dx\ge f_m\,dx=\mu_2[u_m],
 \qquad
 \mu_2[u_m]=f_m\,dx\ge b_m^2f\,dx=\mu_2[U_m^+].
\]
Since \(u\le S\),
\begin{equation}
 U_m^-\le\phi_m\le U_m^+\quad\hbox{on }\partial D.
\label{eq:barrier-boundary}
\end{equation}
We shall use the comparison theorem \cite[Theorem~3.1]{TW1997}: if \(p,q\in C(\overline D)\cap\Phi_2(D)\),
\[
 \mu_2[p]\ge\mu_2[q]\quad\hbox{in }D,
 \qquad p\le q\quad\hbox{on }\partial D,
\]
then \(p\le q\) in \(D\).
Applying the theorem first to \((U_m^-,u_m)\) and then to \((u_m,U_m^+)\) yields
\[
 U_m^-\le u_m\le U_m^+.
\]
Consequently
\begin{equation}
 \norm{u_m-u}_{L^\infty(D)}
 \le\eta_m\osc_Du+e_m\longrightarrow0.
\label{eq:uniform-convergence}
\end{equation}
In particular, \(u_m\) are uniformly bounded in \(L^\infty(D)\).

Choose \(D'\Subset D''\Subset D\).
The already proved smooth theorem, applied on balls whose fixed dilates stay inside \(D\), together with \eqref{eq:approximation-data} and \eqref{eq:uniform-convergence}, gives
\[
 \sup_m\norm{u_m}_{C^{2,\alpha}(D'')}\le C.
\]
For any \(0<\beta<\alpha\), compact embedding gives a subsequence converging in \(C^{2,\beta}(\overline{D'})\).
Uniform convergence \eqref{eq:uniform-convergence} identifies its limit with \(u\).
Finally, for fixed \(x,y\in D'\), pass to the limit in
\[
 \abs{D^2u_m(x)-D^2u_m(y)}
 \le C\abs{x-y}^\alpha.
\]
Thus \(D^2u\in C^\alpha(D')\) with the same bound.
Since \(D'\Subset D''\Subset D\Subset B_2\) were arbitrary, \(u\in C^{2,\alpha}_{\rm loc}(B_2)\).
\end{proof}

\bigskip

\noindent \textbf{Declaration on the use of AI: } The authors acknowledge the assistance of AI in refining certain arguments, checking technical details, and helping complete the final proof.
The authors have reviewed and verified all AI-assisted material and take full responsibility for the content of this paper.

\section*{Acknowledgement}
We sincerely thank Runze Li and Chuyu Liao for their helpful discussions and valuable suggestions at an early stage of this work.
Without their help, this article might not have been completed by this point.

\end{document}